\documentclass[11pt]{amsart}

\usepackage{amsfonts, amsmath,amscd, amssymb, latexsym, mathrsfs, slashed, stmaryrd, verbatim, wasysym }
\usepackage{graphicx}
\usepackage[all]{xy}
\usepackage{slashed}

\newtheorem*{acknowledgments}{Acknowledgments}
\newtheorem{theorem}{Theorem}
\newtheorem{corollary}[theorem]{Corollary}
\newtheorem{definition}[theorem]{Definition}

\newtheorem{lemma}[theorem]{Lemma}
\newtheorem{proposition}[theorem]{Proposition}
\newtheorem{remark}[theorem]{Remark}

\numberwithin{equation}{section}
\numberwithin{theorem}{section}

\newcommand{\sign}{\operatorname{sign}}
\newcommand{\obf}{\operatorname{bf}}
\newcommand{\ocf}{\operatorname{cf}}
\newcommand{\otf}{\operatorname{tf}}

\newcommand{\bhs}[1]{\mathfrak B_{#1}}
\renewcommand{\tilde}{\widetilde}
\renewcommand{\bar}{\overline}
\renewcommand{\Re}{\operatorname{Re}}
\renewcommand{\hat}[1]{\widehat{#1}}

\newcommand{\wt}[1]{\widetilde{#1}}
\newcommand{\rest}[1]{\big\rvert_{#1}} 

\newcommand\pa{\partial}

\newcommand{\db}{\bar{\partial}}

\newcommand{\SL}{\operatorname{SL}}

\newcommand\cf{cf\@. }

\newcommand\eps\varepsilon
\newcommand\p\phi
\renewcommand\P\Phi
\renewcommand\det{\operatorname{det}}

\newcommand{\Cl}{\mathbb{C}l}

\newcommand\fD{\operatorname{\phi-hc}}
\newcommand\hc{\operatorname{hc}}

\newcommand\CI{{\mathcal{C}}^{\infty}}

\newcommand{\lrpar}[1]{\left( #1 \right)}
\newcommand{\lrspar}[1]{\left[ #1 \right]}

\newcommand{\norm}[1]{\lVert #1 \rVert}

\newcommand\Ch{\operatorname{Ch}}

\newcommand\coker{\operatorname{coker}}
\newcommand\Det{\operatorname{Det}}
\newcommand\diag{\operatorname{diag}}

\newcommand\End{\operatorname{End}}

\DeclareMathOperator*{\FP}{\operatorname{FP}}

\newcommand\Id{\operatorname{Id}}

\newcommand\RTr[1]{{}^R\operatorname{Tr}\left( #1 \right)}

\newcommand\Str{\operatorname{Str}}
\newcommand\Td{\operatorname{Td}}

\newcommand\Tr{\operatorname{Tr}}
\newcommand\STr{\operatorname{STr}}
\newcommand\tr{\operatorname{tr}}

\newcommand\PU{\operatorname{PU}}

\newcommand\bbA{\mathbb{A}}
\newcommand\bbB{\mathbb{B}}
\newcommand\bbC{\mathbb{C}}

\newcommand\bbH{\mathbb{H}}

\newcommand\bbN{\mathbb{N}}

\newcommand\bbR{\mathbb{R}}

\newcommand\bbZ{\mathbb{Z}}

\newcommand\cE{\mathcal{E}}

\newcommand\cH{\mathcal{H}}

\newcommand\cK{\mathcal{K}}

\newcommand\cN{\mathcal{N}}

\newcommand\cP{\mathcal{P}}

\newcommand\cR{\mathcal{R}}
\newcommand\cS{\mathcal{S}}
\newcommand\cT{\mathcal{T}}
\newcommand\cU{\mathcal{U}}
\newcommand\cV{\mathcal{V}}

\DeclareMathAlphabet{\mathpzc}{OT1}{pzc}{m}{it}

\renewcommand{\Im}{\operatorname{Im}}
\renewcommand{\Re}{\operatorname{Re}}

\newcommand\cSigma{\overline{\Sigma}}
\newcommand\tSigma{\widetilde{\Sigma}}
\newcommand\pardeg{\operatorname{par}\operatorname{deg}}
\newcommand\rank{\operatorname{rank}}
\newcommand\cbbH{\overline{\bbH}}
\newcommand\tbbH{\widetilde{\bbH}}
\newcommand\U{\operatorname{U}}
\newcommand\tE{\widetilde{E}}
\newcommand\Ld{\operatorname{L}}
\newcommand\tcE{\widetilde{\cE}}
\newcommand\spec{\operatorname{spec}}
\newcommand\GL{\operatorname{GL}}
\newcommand\ad{\operatorname{ad}}
\newcommand\Ad{\operatorname{Ad}}

\begin{document}
\title[Families index for parabolic bundles]{Some index formul\ae{} on the moduli space of
stable parabolic vector bundles}
\author{Pierre Albin}
\author{Fr\'ed\'eric Rochon}
\address{Department of Mathematics, University of Illinois at Urbana-Champaign}
\email{palbin@illinois.edu} 
\address{Department of Mathematics, Australian National University}
\email{frederic.rochon@anu.edu.au} 
\thanks{The first author was partially supported by a NSF postdoctoral fellowship. The second author was supported by a NSERC discovery grant.}

\newcommand\datver[1]{\def\datverp%
 {\par\boxed{\boxed{\text{Version: #1; Run: \today}}}}}
\datver{0.9; Revised: September 8, 2010}

\newcommand\Sec[1]{\cS_{ #1 }} 
\newcommand\Curv[1]{\hat \cR_{ #1 }} 

\begin{abstract}
We study natural families of $\db$-operators on the moduli space of stable parabolic vector bundles.
Applying a families index theorem for hyperbolic cusp operators from our previous work, we find  formul\ae{} for the Chern characters of the associated index bundles. 
The contributions from the cusps are explicitly expressed in terms of the Chern characters of natural vector bundles related to the parabolic structure.
We show that our result implies formul\ae{} for the Chern classes of the associated determinant bundles consistent with a result of Takhtajan and Zograf.
\end{abstract}

\maketitle

\section*{introduction}

The Atiyah-Singer index theorem has various generalizations on non-compact manifolds and manifolds with
boundary, among the most famous being the Atiyah-Bott index theorem \cite{Atiyah-Bott} for local elliptic boundary conditions and the Atiyah-Patodi-Singer index theorem \cite{APS} for global elliptic boundary conditions or cylindrical ends. 
Part of the reason explaining the great variety of 
possible generalizations is that, on a non-compact manifold, the index of an elliptic operator depends
in a subtle way on its behavior at infinity.  
For instance, in contrast to closed manifolds, elliptic operators are not necessarily Fredholm.
Some extra conditions have to be satisfied at infinity, the precise conditions depending on which
type of operators one considers.  

Often, the non-compact manifold is described as the
interior of a manifold with boundary (or a manifold with corners) so that the behavior of the operator 
at infinity is encoded by a symbol defined on the boundary -- usually called the normal operator or the indicial 
family.  A corresponding calculus of pseudodifferential operators in which one can construct parametrices then  allows one to identify the elliptic operators which are Fredholm as those whose normal operator is invertible, intuitively an ellipticity condition at infinity.  Over the years, various types of 
pseudodifferential calculi have been introduced on non-compact manifolds , for instance  
the b-calculus \cite{Melrose-Mendoza}, \cite{APSbook}, the $\Theta$-calculus 
\cite{Epstein-Melrose-Mendoza}, the $0$-calculus \cite{MazzeoMelrose0}, the 
edge calculus \cite{Mazzeo}, the scattering calculus \cite{MelroseGST} or  the fibred cusp calculus \cite{MazzeoMelrose}.  These calculi are usually associated to the asymptotic behavior at infinity of some
complete Riemannian metric in the sense that the associated Laplacian or Dirac type operator is an element
of the calculus.  More generally, one can talk about a Lie structure at infinity \cite{MelroseLie} and a general procedure to get a corresponding pseudodifferential calculus has been obtained by Ammann, Lauter and Nistor \cite{ALN} using groupoids.

We are interested in the situation where the non-compact manifold considered is a Riemann surface
$\Sigma= \cSigma\setminus\{p_{1},\ldots,p_{n}\}$ of genus $g$ with $n$ punctures.  Provided 
$2g-2+n>0$, such a Riemann surface admits a canonical hyperbolic metric $g_{\Sigma}$ with conformal class prescribed by the complex structure.  Near infinity, that is, near a puncture, the geometry associated
to such a metric is the one of a cusp.  The pseudodifferential operators one gets out of this 
geometry at infinity are called $\hc$-operators (or $d$-operators in the terminology of Vaillant
\cite{Vaillant}).  In his thesis \cite{Vaillant}, Vaillant studied in great detail the Dirac type 
operators associated to geometries that asymptotically behave like a cusp or a fibred cusp on a non-compact manifold
(not necessarily a punctured Riemann surface).  He gave an explicit description of the continuous 
spectrum in terms of an operator at infinity and, in the Fredholm case, provided an index formula  
involving the usual Atiyah-Singer term in the interior and some eta forms defined in terms of the 
normal operator.

In general, eta forms are very hard to compute.  However, on punctured Riemann surfaces, the 
geometry at infinity is sufficiently simple to allow explicit computations.  In \cite{Albin-Rochon2},
using the generalization of Vaillant's index theorem to families from \cite{Albin-Rochon1},
this fact was put to use to get a local index theorem in terms of the Mumford-Morita-Miller classes for families of $\db$-operators 
parametrized by the moduli space of Riemann surfaces of genus $g$ with $n$ marked points.  
Using heat kernel techniques as in \cite{BF} (see also \S9-10 in \cite{BGV}), it was also possible 
to give an alternate proof of the formula of Takhtajan and Zograf \cite{TZ} for the curvature
of the Quillen connection defined on the corresponding determinant line bundle.

In this paper, we want to apply the generalization of Vaillant's index theorem to families
(Theorem 4.5 in \cite{Albin-Rochon1}) to another moduli space, namely, the moduli space of stable
parabolic vector bundles with vanishing parabolic degree on a Riemann surface $\cSigma$ with marked points. 
A parabolic vector bundle is a holomorphic vector bundle on $\cSigma$ together with a parabolic structure
specified at each marked point by a set of weights and multiplicities.  According to 
\cite{Boden-Yokogawa}, for generic weight systems, the moduli space $\cN$ of stable parabolic vector
bundles of vanishing parabolic degree admits a universal parabolic vector bundle, that is, a holomorphic
vector bundle $E\to \cSigma\times \cN$ such that for each $m\in \cN$,  $\left. E\right|_{\cSigma\times \{m\}}$ represents
the parabolic vector bundle described by the point $m$.  In this context, one gets a family
of $\db$-operators $\db_{E}$ parametrized by $\cN$ acting fibrewise on sections of
$\left. E\right|_{\cSigma\times \{m\}}$ above $m\in \cN$.  Similarly, one gets a family of 
$\db$-operators for the endomorphism bundle $\cE:= \End(E)$.  This family is of particular 
importance for the moduli space $\cN$ since there is a canonical identification of the 
tangent bundle of $\cN$ with the cokernel bundle $\coker\db_{\cE}\to \cN$.
  
In \cite{TZ2}, Takhtajan and Zograf were able to define the Quillen connection on the determinant
line bundle of the family $\db_{\cE}$ using an appropriate Selberg Zeta function.  More importantly,
they obtained an explicit formula for its curvature, and consequently for the first Chern form of the tangent
bundle of $\cN$.  Their formula involves the usual Atiyah-Singer term and a 
cuspidal defect.

In this paper we show that, as operators on a punctured Riemann surface $\Sigma$, $\db_{E}$ and $\db_{\cE}$ are smooth families of Fredholm Dirac type $\hc$-operators, so that the families index formula of \cite{Albin-Rochon1} applies to them.  We then perform a computation of the eta forms to express them in terms of explicit data coming from the parabolic structure.  This leads to local families index formul\ae{} for the families $\db_{E}$ and 
$\db_{\cE}$ (theorem~\ref{if.17} and theorem~\ref{if.28}).  In both cases, we are 
also able, as in \cite{Albin-Rochon2}, to define the Quillen metric on the determinant line bundle via
heat kernel techniques and identify its curvature with the $2$-form part of the index formula
(theorem~\ref{ds.3} and theorem~\ref{dlb.11}).  For the determinant line of the family $\db_{\cE}$,
our formula agrees with the one of Takhtajan and Zograf (Theorem 2 in \cite{TZ2}).

The paper is organized as follows.  In \S\ref{set.0}, we review the definitions and the various properties
surrounding the notion of parabolic vector bundle.  In \S\ref{db.0}, we describe how the $\db$-operator 
associated to a parabolic vector bundle can be seen as a Dirac type $\hc$-operator and we use the
criterion of Vaillant to check that it is Fredholm.  In \S\ref{ms.0}, we consider  families of such 
operators parametrized by the moduli space of stable parabolic vector bundles of parabolic degree
equal to zero.  In particular, following \cite{TZ2}, we describe the natural connection that can be put on the universal 
parabolic vector bundle $E\to \Sigma \times \cN$.  In \S\ref{if.0}, we compute explicitly the eta
forms involved in the index formul\ae{} and get our main results, theorem~\ref{if.17} and theorem~\ref{if.28}.
Finally, in \S~\ref{dlb.0}, we define the Quillen connection of the determinant line bundle of the
families $\db_{E},\db_{\End(E)}$ and compute their curvature.  In Appendix~\ref{appendixA}, we also quickly indicate how the short time asymptotic of the
regularized trace of the heat kernel can be deduced from the work of Vaillant \cite{Vaillant} and the pushforward theorem.

\begin{acknowledgments}
The authors are grateful to MSRI for its hospitality and support while part
of this work was conducted.  The authors wish also to thank 
Ben Brubaker, Daniel Grieser, Johan Martens, Rafe Mazzeo, Richard Melrose and Leon Takhtajan for helpful discussions.
\end{acknowledgments}

\section{Stable parabolic vector bundles} \label{set.0}

Let $\cSigma$ be a compact Riemann surface of genus $g$ and negative Euler characteristic (i.e., $g \geq 2$).
It is well-known that a representation of the fundamental group
\begin{equation*}
	\rho: \pi_1(\cSigma) \to \U(k)
\end{equation*}
gives rise to a holomorphic vector bundle over $\cSigma$, $E_\rho$. Indeed, since $\bbH$ is the universal covering space of $\cSigma$, it suffices to take the product $\bbH \times \bbC^k$ and mod out by the action of $\pi_q(\cSigma)$ induced by $\rho$. Furthermore, the trivial metric and connection on the bundle $\bbH \times \bbC^k \to \bbH$ descend to a Hermitian metric and compatible flat holomorphic connection on $E_\rho$. 
This connection determines $\rho$ as its holonomy representation; which shows that only vector bundles arising from complex representations of the fundamental group admit flat holomorphic connections.

It is also possible to give a criterion in terms of geometric invariant theory to describe what kind
of holomorphic vector bundles arises in this way.  If $E$ is a holomorphic vector bundle over $\cSigma$,
its slope is the quotient 
\begin{equation}
   \mu(E):= \frac{\deg(E)}{\rank(E)}
\label{slope}\end{equation}
where $\deg(E)= c_{1}(E)[\cSigma]$ is the degree of $E$.  The vector bundle $E$ is said to be stable
if whenever $F$ is a holomorphic sub-bundle, we have
\[
       \mu(F)<\mu(E).
\]
A theorem of Narasimhan and Seshadri \cite{Narasimhan-Seshadri} asserts (among other things) that a holomorphic vector
bundle $E$ of degree zero is stable if and only if it is isomorphic to a vector bundle induced from an
\textbf{irreducible} representation $\rho: \pi_{1}(\cSigma)\to \U(k)$ where $k$ is the rank of $E$.  
In particular, a holomorphic vector bundle $E$ is of the form $E_{\rho}$ for some unitary
representation $\rho$ if and only if it is a direct sum of stable vector bundles of degree zero.

Donaldson \cite{Donaldson} gave a new proof of the theorem of Narasimhan and Seshadri using gauge theory. He found a geometric interpretation of stability by showing that a holomorphic bundle $E$ over $\cSigma$ is stable if and only if there is a unitary connection on $E$ having constant central curvature $R = -2\pi i \mu(E) (*1)$. Since this implies that stable vector bundles of degree zero admit a flat connection, the
theorem of Narasimhan and Seshadri follows immediately.

Next consider the situation where the Riemann surface $\cSigma$ has some marked points.  Precisely,
let $\cSigma$ be a compact Riemann surface of genus $g$ and $S = \{ p_1, \ldots, p_n\}$ a subset consisting of $n$ distinct points.
Assume that the Euler characteristic of the punctured Riemann surface
$\Sigma:= \cSigma\setminus S$ (equal to $2-2g-n$) is negative.
By the uniformization theorem for Riemann surfaces, one can represent the punctured Riemann
surface $\Sigma$ as the quotient 
\begin{equation}
     \Sigma = \Gamma \setminus \bbH, \quad \Gamma\cong \pi_{1}(\Sigma) 
\label{set.6}\end{equation}
of the hyperbolic half-plane $\bbH:=\{z\in \bbC\, | \, \Im(z)>0\}$ by 
a torsion-free Fuchsian group $\Gamma$ generated by hyperbolic transformations
$A_{1},B_{1},\cdots, A_{g},B_{g}$ and parabolic transformations $S_{1},\ldots,S_{n}$
satisfying the single relation
\begin{equation}
   A_{1}B_{1}A^{-1}_{1}B^{-1}_{1}\cdots A_{g}B_{g}A^{-1}_{g}B^{-1}_{g} S_{1}\cdots S_{n}=1.
\label{set.7}\end{equation}

Let $x_{1},\ldots,x_{n}$ be the fixed points of $S_{1},\ldots, S_{n}$ and let $\overline{\bbH}$ be the union
of $\bbH$ with the set of all points fixed by a parabolic element of $\Gamma$.  As pointed out in \cite{TZ2}, the action of $\Gamma$
naturally extends to $\overline{\bbH}$ so that $\cSigma\cong \Gamma\setminus \overline{\bbH}$ and   the
image of $x_{1},\ldots,x_{n}$ under the quotient map are precisely the marked points $p_{1},\ldots,p_{n}$
on $\cSigma$.  For the fixed point $x_{i}$ of $S_{i}$, it is convenient to choose
an element $\sigma_{i}\in \SL(2,\bbR)$ such that $\sigma_{i}(\infty)= x_{i}$ and
\[
\sigma_{i}^{-1}S_{i}\sigma_{i}= \left(\begin{array}{cc} 1 & \pm 1 \\
                                                       0 & 1 
                                     \end{array} \right). 
\]
This provides a local coordinate $\zeta_{i}= e^{2\pi i \sigma_{i}^{-1}z}$ near $p_{i}\in \Gamma\setminus \overline{\bbH}$.

Let $\rho: \Gamma\to \U(k)$ be a unitary representation, where $\U(k)$ is the group
of $k\times k$ unitary matrices.  Using
the representation $\rho$, we can define an action of the group $\Gamma$ on the trivial vector bundle $\bbH\times \bbC^{k}$ by
\begin{equation}
    \begin{array}{llcl}
      \gamma: & \bbH\times \bbC^{k} & \to & \bbH\times \bbC^{k} \\
              & (z,v) & \mapsto & (\gamma z, \rho(\gamma)v)  
    \end{array}
\label{set.9}\end{equation}
The quotient of this action define a flat Hermitian bundle $E_{\rho}$ on the punctured Riemann surface
$\Sigma:= \cSigma\setminus\{p_{1},\ldots,p_{n}\}$.  As described in (\cite{Mehta-Seshadri}, definition 1.1), one can also obtain a holomorphic vector bundle 
on $\cSigma$ by considering the sheaf $\cV$ of $\Gamma$-invariant holomorphic sections
of $\bbH\times \bbC^{k}$ which are bounded at the cusps.  This is  playing  the r\^ole of the sheaf of 
$\Gamma$-invariant holomorphic sections of $\cbbH\times \bbC^{k}$.  The direct image of this sheaf under
the canonical map $p^{\Gamma}:\cbbH\to \cSigma$ is locally free of rank $k$, so defines
a holomorphic vector bundle
\begin{equation}
  \overline{E}_{\rho}\to \cSigma
\label{set.10}\end{equation}
of rank $k$.  This is a typical example of a parabolic vector bundle, the definition of which we now recall.

\begin{definition}
A \textbf{parabolic structure} on a holomorphic vector bundle $\pi:E\to \cSigma$ consists in giving at
each point $p\in S$ 
\begin{itemize}
\item a flag $E_{p}= F_{1}E_{p}\supset F_{2}E_{p}\supset \cdots \supset F_{r(p)}E_{p}\supset F_{r(p)+1}E_{p}=\emptyset$,
\item weights $\alpha_{1}(p),\cdots, \alpha_{r(p)}(p)$ associated to $F_{1}E_{p},\ldots, F_{r(p)}E_{p}$ in such a way that $0\le \alpha_{1}(p)<\alpha_{2}(p)<\cdots <\alpha_{r(p)}(p)<1$.
\end{itemize}
The \textbf{multiplicity} of the weight $\alpha_{i}(p)$ is $k_{i}(p):=\dim F_{i}E_{p} - \dim F_{i+1}E_{p}$.
A \textbf{parabolic vector bundle} is a holomorphic vector bundle equipped with a 
parabolic structure.
\label{set.1}\end{definition}

For the bundle $\overline{E}_{\rho}$, the  natural parabolic structure is specified at $p_{i}\in S$ by the eigenspaces and the 
eigenvalues of $\rho(S_{i})$ where $S_{i}$ is the parabolic element fixing the cusp associated
to $p_{i}$.  Indeed, let 
\begin{equation}
 \bbC^{k}= \bigoplus_{j=1}^{r(p_{i})} \overline{E}_{p_{i}}^{j}
\label{set.11}\end{equation}
be a decomposition of $\bbC^{k}$ in terms of the eigenspaces $\overline{E}^{j}_{p_{i}}$ of
$\rho(S_{i})$ with eigenvalue $\lambda_{j}(p_{i})= e^{2\pi i \alpha_{j}(p_{i})}$.  Assume
that they are ordered in such a way that $0\le \alpha_{1}(p_{i})<\alpha_{2}(p_i)<\cdots <
\alpha_{r_{i}}(p_{i})<1$.  Then the parabolic structure at $p_i$ is given by 
\begin{equation}
   F_{j}E_{p_{i}}:= \bigoplus_{m=j}^{r_{i}} \overline{E}_{p_{i}}^{m}, \quad \mbox{with weight}                       \; \alpha_{j}(p_{i}),
\label{set.12}\end{equation}
where, as described in (\cite{Mehta-Seshadri}, p.208), the identification of 
$\overline{E}_{\rho}$ with the trivial bundle $\underline{\bbC}^{k}$ is given near $p_{i}$ 
by
\begin{equation}
    \begin{array}{ccc}
       \bigoplus_{j=1}^{r_{i}} \CI(\cU;\overline{E}^{j}_{p_{i}}) & \to & \CI(\cU;\overline{E}_{\rho}) \\
        (\sigma_{1},\ldots,\sigma_{r_{i}})& \mapsto & \sum_{j=1}^{r_{i}} 
          \zeta_{i}^{\alpha_{j}(p_{i})} \sigma_{j}
    \end{array}
\label{set.13}\end{equation}
with $\zeta_{i}= e^{2\pi i \sigma_{i}^{-1}z}$ the complex coordinate introduced earlier.  Since $\rho$ is a unitary representation, the bundle
$\overline{E}$ comes with a natural Hermitian metric $h_{E}$ when restricted to the punctured Riemann
surface $\Sigma$.  This Hermitian metric extend to a Hermitian metric $h_{\overline{E}}$ on 
$\cSigma$ which degenerates at the punctures $p_{1},\ldots,p_{n}$.  In the trivialization \eqref{set.13},
it takes the form
\begin{equation}
  h_{\overline{E}}\left( \sum_{j=1}^{r_{i}} 
          \zeta_{i}^{\alpha_{j}(p_{i})} \sigma_{j}, \sum_{j=1}^{r_{i}} 
          \zeta_{i}^{\alpha_{j}(p_{i})} \sigma_{j}\right) =
           \sum_{j=1}^{r_{i}} |\zeta_{i}|^{2\alpha_{j}(p_{i})} |\sigma_{j}|^{2}
\label{set.13b}\end{equation}

The {\bf parabolic degree} of a parabolic bundle is defined by
\begin{equation*}
	\pardeg(E) = \deg(E) + \sum_{p \in S} \sum_{j=1}^{r(p)} \alpha_j(p) k_j(p),
\end{equation*}
and its {\bf parabolic slope} (again denoted $\mu(E)$) is the ratio of its parabolic degree and its rank.
A holomorphic sub-bundle of $E$ with its induced parabolic structure is known as a parabolic sub-bundle of $E$, and $E$ is said to be a {\bf stable parabolic bundle} if the parabolic slope of any parabolic sub-bundle is strictly smaller than the parabolic slope of $E$.
The Mehta-Seshadri theorem \cite{Mehta-Seshadri} says that a parabolic vector bundle over $\cSigma$ arises from an irreducible unitary representation of $\pi_1(\cSigma \setminus S)$ if and only if it is parabolically stable and has vanishing parabolic degree.


Biquard \cite{Biquard} proved the analogue of Donaldson's theorem for parabolic bundles, namely, that a parabolic bundle is parabolically stable precisely when it admits a unitary connection with curvature $R = -2\pi i \mu(E) (*1)$, hence recovering the theorem of Mehta-Seshadri when $\mu(E) = 0$.  Notice that while Mehta and Seshadri worked with rational weights, Biquard's proof works for arbitrary real weights.

From the Mehta-Seshadri theorem and its generalization by Biquard, the moduli space $\cN$ of stable parabolic vector bundles
of rank $k$ and parabolic degree zero with prescribed weights and multiplicities at $p_{1},\ldots
p_n\in \cSigma$ is given by 
\begin{equation}
    \cN= \hom( \Gamma, \U(k))^{0}/ U(k)
\label{set.15}\end{equation}  
where $\hom(\Gamma,\U(k))^{0}$ is the space of irreducible admissible representations 
$\Gamma\to \U(k)$ for the prescribed weights and multiplicities with $\U(k)$ acting on this
space by conjugation.  Recall that a representation $\rho:\Gamma\to \U(k)$ is
\textbf{admissible} with respect to a system of weights and multiplicities if the 
corresponding parabolic vector bundle $\overline{E}_{\rho}$ has a parabolic structure
compatible with this set of weights and multiplicities.  

Although $\cSigma$ is perhaps the most natural compactification of $\Sigma$, our approach will be to consider a different compactification of $\Sigma$, also natural, to a manifold with boundary. By replacing each marked point in $\cSigma$ with a circle, we keep track of the `direction' of approach to the cusp. In contrast to $\cSigma$, this has the advantage that the natural metric and connection of a stable parabolic vector bundle of degree zero extend non-singularly to the compactification. We will measure regularity in a way adapted to the degeneracy of the hyperbolic metric at the cusps by working with a class of adapted differential operators called `hyperbolic cusp' differential operators.

\section{ The $\db$-operator for stable parabolic vector bundles}\label{db.0}

Fix an irreducible representation $\rho:\Gamma\to \U(k)$ with 
prescribed weights and multiplicities and let $\overline{E}=\overline{E}_{\rho}$ be the corresponding stable parabolic vector bundle of degree zero.  
Since $\overline{E}$ is in particular a holomorphic vector bundle, this 
means there is a $\db$-operator 
\begin{equation}
  \db_{\overline{E}}: \CI(\cSigma;\overline{E})\to \CI(\cSigma;
\Lambda^{0,1}_{\cSigma}\otimes \overline{E}).
\label{db.1}\end{equation}
We are interested in the restriction of this operator to the punctured
Riemann surface $\Sigma$,
\begin{equation}
  \db_{E}: \CI(\Sigma;E)\to \CI(\Sigma; \Lambda^{0,1}_{\Sigma}\otimes E)
\label{db.2}\end{equation}
where $E:= \left. \overline{E}\right|_{\Sigma}$ is the restriction of 
$\overline{E}$ to $\Sigma$.  We are also interested in the $\db$ operator
associated to the endomorphism bundle $\cE:= \End(E)$,
\begin{equation}
\db_{\cE}: \CI(\Sigma;\cE)\to \CI(\Sigma; \Lambda^{0,1}_{\Sigma}\otimes \cE).
\label{db.3}\end{equation}
To study these operators from
the point of view of hyperbolic cusp operators ($\hc$-operators), 
consider the radial blow-up $\tSigma$ of $\cSigma$ at the points
$p_{1},\ldots,p_{n}$ with blow-down map
\begin{equation}
        \beta: \tSigma\to \cSigma.
\label{db.4}\end{equation}
The uniformization theorem for Riemann surfaces specifies a choice of 
boundary defining function $\rho_{\Sigma}$ on $\tSigma$, that is,
a function $\rho_{\Sigma}\in \CI(\tSigma)$ positive in the interior
such that $\left.\rho_{\Sigma}\right|_{\pa\tSigma}\equiv 0$ and with
$d\rho_{\Sigma}$ nowhere zero on the boundary.  Namely, let 
$g_{\Sigma}$ be the canonical hyperbolic metric in the 
conformal class specified by the complex structure.  Near a point 
$p\in S$, choose a complex coordinate $\zeta:= e^{2\pi i z}$ on
$\cSigma$ with $\zeta(p)=0$ such that 
\begin{equation}
   g_{\Sigma}= \frac{ dx^{2}+ dy^{2}}{y^{2}}, \quad z=x+iy
\label{db.5}\end{equation} 
in this coordinate.  One can use instead the polar coordinates
\begin{equation}
 \theta = x, \quad  r= \frac{1}{y}
\label{db.6}\end{equation}
which also make sense on $\tSigma$ with $r=0$ corresponding to the
boundary component $\pa \tSigma_{p}:= \beta^{-1}(p)$.  Near 
$\pa \tSigma_{p}$, we choose the boundary defining function $\rho_{\Sigma}$
to be given by
\begin{equation}
   \rho_{\Sigma}(\theta,r)=r.
\label{db.7}\end{equation} 
Doing this near each boundary component and extending $\rho_{\Sigma}$ to the
interior as a positive function, we get the desired boundary defining 
function.  

Let $\tbbH$ be the universal cover of $\tSigma$ so that 
\begin{equation}
  \tSigma= \Gamma \setminus \tbbH
\label{db.8}\end{equation}
where we use the canonical identification $\pi_{1}(\tSigma)=\pi_{1}(\Sigma)=\Gamma$.  On $\tbbH$ there is a `blow-down' map $\beta_{\bbH}:\tbbH\to \cbbH$ 
compatible with the action of $\Gamma$.  Since the action of $\Gamma$ is 
free on $\tbbH$, this means that $\tbbH$ is in a sense a free resolution
of the action of $\Gamma$ on $\cbbH$.  In particular, on $\tSigma$, 
the pulled back vector bundle $\tE:=\beta^{*}\overline{E}$ 
can be directly described as a quotient of the trivial vector bundle
$\tbbH\times \bbC^{k}$ on $\tbbH$,
\begin{equation}
  \tE:= \Gamma\setminus \left(\tbbH\times \bbC^{k}  \right)
\label{db.9}\end{equation} 
with $\gamma\in \Gamma$ acting on $\tbbH\times \bbC^{k}$ by
\begin{equation}
  \begin{array}{lccc}
   \gamma: & \tbbH\times \bbC^{k} & \to & \tbbH\times \bbC^{k} \\
           & (z,v) & \mapsto & (\gamma z, \rho(\gamma)v).
  \end{array}
\label{db.10}\end{equation}
Since $\rho$ is a unitary representation, the canonical Hermitian metric on
$\tbbH\times \bbC^{k}$ descends to a Hermitian metric $h_{\tE}$ on 
$\tE$.  With this metric and the corresponding Chern connection,
the vector bundle $\tE$ becomes a flat Hermitian vector bundle with
holonomy specified by the representation $\rho$.  In particular,
 it is locally isomorphic to the trivial Hermitian bundle 
$\underline{\bbC}^{k}$.  Under such a local identification near the boundary, 
the operator $\db_{\tE}$ can be written as
\begin{equation}
   \db_{\tE}= \frac{1}{2}\left( \rho_{\Sigma} d\theta + i\frac{d\rho_{\Sigma}}{\rho_{\Sigma}}\right) \left( \frac{1}{\rho_{\Sigma}}\frac{\pa}{\pa \theta}-
i \rho_{\Sigma}\frac{\pa}{\pa \rho_{\Sigma}}\right)
\label{db.11}\end{equation}
in the polar coordinates $(r,\theta)$, \cf equation (3.3) in \cite{Albin-Rochon2} .
The section $\left( \rho_{\Sigma} d\theta + i\frac{d\rho_{\Sigma}}{\rho_{\Sigma}}\right)$ of $\Lambda^{0,1}_{\Sigma}$ becomes singular as one approaches the
boundary.  However, it is smooth up to the boundary as a section
of ${}^{\hc}\Lambda^{0,1}_{\tSigma}$, which is defined to be 
the $(0,1)$ part of the
complexified hyperbolic cusp cotangent bundle
\begin{equation}
  {}^{\hc}T^{*}\tSigma\otimes_{\bbR} \bbC = 
{}^{\hc}\Lambda^{1,0}_{\tSigma} \oplus {}^{\hc}\Lambda^{0,1}_{\tSigma},
\label{db.12}\end{equation}
 where the  bundle ${}^{\hc}T^{*}\tSigma$ is defined in such a way that there is 
a canonical identification 
\begin{multline}
\CI(\tSigma; {}^{\hc}T^{*}\tSigma)= 
\{ s \in \CI(\tSigma;T^{*}\tSigma) \quad | \quad 
  \exists C>0  \, \mbox{such that} \\
       g_{\Sigma}(s(z),s(z))\le C \; \forall \, z\in \Sigma \}.
\label{db.13}\end{multline}
In that way, $\db_{E}$ can be seen as a $\hc$-operator ($d$-operator
in the terminology of Vaillant \cite{Vaillant}),
\begin{equation}
  \db_{E}: \CI( \tSigma; \tE) \to \frac{1}{\rho_{\Sigma}}\CI(\tSigma;
   {}^{\hc}\Lambda^{0,1}\tSigma\otimes \tE).
\label{db.14}\end{equation}
To keep the discussion short, we refer the reader to
the first three sections of \cite{Albin-Rochon2} for a quick review 
of $\hc$-operators and a similar construction, and to
\cite{Vaillant} and \cite{Albin-Rochon1} for further details.  

Similarly,
 the operator $\db_{\cE}$ can be seen as a $\hc$-operator
\begin{equation}
  \db_{\cE}: \CI( \tSigma; \tcE) \to \frac{1}{\rho_{\Sigma}}\CI(\tSigma;
   {}^{\hc}\Lambda^{0,1}\tSigma\otimes \tcE).
\label{db.15}\end{equation}
where $\tcE:= \End(\tE)$.
The metric $g_{\Sigma}$ induces  Hermitian metrics on 
$\Lambda^{0,1}_{\Sigma}$ and on ${}^{\hc}\Lambda^{0,1}_{\tSigma}$.   
Together with the natural Hermitian metrics $h_{\tE}$ and 
$h_{\tcE}$, this therefore defines Hilbert spaces
$\cH_{\tE,j}:=\Ld^{2}(({}^{\hc}\Lambda^{0,1})^{j}\otimes\tE)$ and 
$\cH_{\tcE,j}:=\Ld^{2}(({}^{\hc}\Lambda^{0,1})^{j}\otimes\tcE)$ 
for $j\in\{0,1\}$ with inner product given
by
\begin{equation}
\begin{gathered}
\langle f_{1},f_{2}\rangle_{\cH_{\tE,j}}: = 
\int_{\tSigma} \langle f_{1}(z),f_{2}(z)\rangle_{({}^{\hc}\Lambda^{0,1})^{j}\otimes \tE} \;dg_{\Sigma}(z) ; \\
\langle f_{1},f_{2}\rangle_{\cH_{\tcE,j}}: = 
\int_{\tSigma} \langle f_{1}(z),f_{2}(z)\rangle_{({}^{\hc}\Lambda^{0,1})^{j}\otimes \tcE} \; dg_{\Sigma}(z) ; 
\end{gathered}
\label{db.16}\end{equation} 
where $dg_{\Sigma}$ is the natural extension of the volume form
of $g_{\Sigma}$ to $\tSigma$.  With these inner products, we 
can define the formal adjoints $\db^{*}_{E}$ and
$\db^{*}_{\cE}$ of $\db_{E}$ and $\db_{\cE}$.  Recall then that the 
operators
\begin{equation}
   D_{E}:= \sqrt{2}(\db_{E}+\db^{*}_{E}), \quad 
   D_{\cE}:= \sqrt{2}(\db_{\cE}+\db^{*}_{\cE})
\label{db.17}\end{equation}
can be interpreted as Dirac type $\hc$-operators  acting on the 
Clifford modules $(\underline{\bbC}\oplus {}^{\hc}
\Lambda^{0,1}_{\Sigma})\otimes \tE$ and 
$(\underline{\bbC}\oplus {}^{\hc}
\Lambda^{0,1}_{\Sigma})\otimes \tcE$ with Clifford action given by
\begin{equation}
   c(f)= \sqrt{2}( \varepsilon(f^{0,1})- \iota(f^{1,0})), 
\quad f\in \CI(\tSigma;{}^{\hc}\Lambda_{\tSigma})
\label{db.18}\end{equation}
 where $\varepsilon(f^{0,1})$ denotes exterior multiplication
by the form $f^{0,1}$.  In his thesis, Vaillant provided a criterion to
determine when a Dirac type $\hc$-operator $\eth_{\hc}$ is Fredholm. 
For a Dirac type $\hc$-operator $\eth_{\hc}$ acting on  $\CI(\tSigma;W)$, the criterion
is relatively easy to formulate.  One first needs to introduce the 
vertical operator
\begin{equation}
    \eth_{\hc}^{V}:= \left. \rho_{\Sigma}\eth_{\hc}\right|_{\pa \tSigma}
\label{db.19}\end{equation}    
on the boundary $\tSigma$.  When $\ker \eth_{\hc}^{V}$ is non-trivial, one can
also introduce a horizontal operator
\begin{equation}
  \eth_{\hc}^{H}: \ker \eth_{\hc}^{V}\to \ker \eth_{\hc}^{V}
\label{db.20}\end{equation}
defined by 
\begin{equation}
   \eth^{H}_{\hc}\xi = \Pi_{0} (\left. \eth_{\hc}\tilde{\xi}\right|_{\pa \tSigma} )
\label{db.21}\end{equation}
where $\xi\in \ker \eth^{V}_{\hc}$, $\tilde\xi\in \CI(\tSigma;W)$ is a smooth
extension of $\xi$ in the interior and $\Pi_{0}$ is the orthogonal 
projection from $\Ld^{2}(\pa\tSigma;W)$ to $\ker \eth^{V}_{\hc}$.  
In terms of these operators, the criterion of Vaillant can be formulated
as follows.
\begin{proposition}[Vaillant \cite{Vaillant}, \S 3]
The continuous spectrum of a  Dirac type self-adjoint operator 
$\eth_{\hc}\in \Psi^{1}_{\hc}(\tSigma;W)$ is governed by the spectrum
of the horizontal operator $\eth^{H}_{\hc}$ with bands of continuous spectrum
starting at the eigenvalues of $\eth^{H}_{\hc}$ and going to infinity.
In particular, $\eth_{\hc}$ is Fredholm if and only if $\eth^{H}_{\hc}$ is
invertible and $\eth_{\hc}$ is Fredholm and  has discrete spectrum whenever $\eth^{V}_{\hc}$ is 
invertible.
\label{db.22}\end{proposition}
To apply this criterion to $D_{E}$ and $D_{\cE}$, we first need to 
describe the restriction of $\tE$ on the boundary $\pa\tSigma$.  Let
us fix a boundary component $\pa \tSigma_{p_{i}}$ for some $p_{i}\in S$.
Identify $\pa \tSigma_{p_{i}}$ with $\bbZ\setminus\bbR$ as oriented 
manifolds. Notice that the orientation of
$\pa\tSigma_{p_{i}}$ induced from $\tSigma$ is such that $\frac{\pa}{\pa u}$ 
is an oriented section of the tangent bundle where $u=-\theta$ in 
terms of the polar coordinates $(r,\theta)$.  
From that perspective, we can 
interpret the restriction $\tE_{i}:= \left.\tE\right |_{\pa\tSigma_{p_{i}}}$
as the quotient 
\begin{equation}
   \bbZ\setminus (\bbR\times \bbC^{k})
\label{db.23}\end{equation}
where $\bbR\times \bbC^{k}$ is the total space of the trivial 
vector bundle $\underline{\bbC}^{k}$ over $\bbR$ on which $\bbZ$ acts by
\begin{equation}
   \begin{array}{lccc}
    m:& \bbR\times \bbC^{k} & \to & \bbR\times \bbC^{k} \\
      &     (u,v)  & \mapsto & (u+m, \rho(S_{i})^{-m}v)
\end{array}
\label{db.24}\end{equation}
for $m\in \bbZ$.  Sections of $\tE_{i}$ then correspond to $\bbZ$-invariant
sections of the trivial bundle $\underline{\bbC}^{k}\to \bbR$.  On the 
other hand, at the boundary component $\pa \tSigma_{p_{i}}$,
the vertical operator of $D_{E}$ is given by
\begin{equation}
   D^{V}_{\tE_{i}}:= \left. \rho_{\Sigma}D_{E}\right|_{\pa \tSigma_{p_{i}}}=
  c(du) \frac{\pa}{\pa u}
\label{db.25}\end{equation}
acting on $(\underline{\bbC}\oplus {}^{\hc}\Lambda^{0,1}_{\tSigma})\otimes
\tE_{i}$.
Under the standard identification
\begin{equation}
-c\left(\frac{d\rho_{\Sigma}}{\rho_{\Sigma}}\right): 
\left.{}^{\hc}\Lambda^{0,1}_{\tSigma}\right|_{\pa \tSigma_{p_{i}}}
\otimes \tE_{i} \to \tE_{i} 
\label{db.26}\end{equation} 
given by Clifford multiplication by the element $-c\left(\frac{d\rho_{\Sigma}}{\rho_{\Sigma}}\right)$, we can rewrite the operator as 
\begin{equation}
 D^{V}_{\tE_{i}}= \left(\begin{array}{cc} 0 & \db^{V}_{\tE_{i}} \\ 
\db_{\tE_{i}}^{V} & 0 \end{array}\right)= \left( \begin{array}{cc}
                        0 &i\frac{\pa}{\pa u}  \\
                        i\frac{\pa}{\pa u}  & 0
                         \end{array} \right)\in
\Psi^{1}(\pa\tSigma_{p_{i}}; \tE_{i}\oplus \tE_{i}).
\label{db.27}\end{equation}
\begin{lemma}
The spectrum of the vertical operator $\db^{V}_{\tE_{i}}$ is given by 
\[
   \lambda_{j,k}= 2\pi(\alpha_{j}(p_{i})+k), \quad j\in\{1,\ldots,r_{i}\}, 
\quad k\in \bbZ
\]
where the eigenvalue $\lambda_{j,k}$ has multiplicity $k_{j}(p_{i})$.
\label{db.28}\end{lemma}
\begin{proof}
Let
\begin{equation}
  \bbC^{k}= \bigoplus_{j=1}^{r_{i}} W_{ij}
\label{db.29}\end{equation}
be the decomposition in terms of the eigenspaces of $\rho(S_{i})$ where 
$W_{ij}$ is the eigenspace corresponding to the eigenvalue 
$e^{2\pi i\alpha_{j}(p_{i})}$. As an operator on
$\bbR$, $\db^{V}_{\tE_{i}}$ commutes with the action of $\bbZ$.  This means
that $\db^{V}_{\tE_{i}}$ decomposes as a sum of operators
\begin{equation}
  \db^{V}_{\tE_{i}}= \bigoplus_{j=1}^{r_{i}} \db^{V}_{\tE_{ij}}
\label{db.30}\end{equation} 
with $\db^{V}_{\tE_{ij}}$ acting on sections of
\begin{equation}
    \tE_{ij}:= \bbZ\setminus (\bbR\times W_{ij}).
\label{db.31}\end{equation}
If $w_{1}^{i},\ldots,w_{k_{j}}^{i}$ is a basis of $W_{ij}$, then clearly
\begin{equation}
  e_{kl}^{ij}= e^{-2\pi i(k+ \alpha_{j}(p_{i}))u}w^{i}_{l}, 
    \quad k\in \bbZ, \quad l\in \{1,\ldots k_{j}(p_{i})\},
\label{db.31b}\end{equation}
is a basis of $\Ld^{2}(\pa\tSigma_{p_{i}};\tE_{ij})$
in terms of eigensections of $\db^{V}_{\tE_{ij}}$ with $e_{kl}^{ij}$
an eigensection with eigenvalue $\lambda_{j,k}$.  Collecting these eigenvalues
for all $j$, the result follows.
\end{proof}
\begin{proposition}
The operator $D_{E}$  is Fredholm. Moreover, if 
none of the weights of the parabolic structure of $\overline{E}$ are zero,
then its spectrum is discrete.
\label{db.32}\end{proposition}
\begin{proof}
This will follows from proposition~\ref{db.22}.  First, if none
of the weights are zero, then it follows from lemma~\ref{db.28} that
the vertical operator $D^{V}_{E}$ is invertible, so that $D_{E}$ is Fredholm
with discrete spectrum.  
If one of the weights is zero at some point, say $\alpha_{1}(p_{i})=0$ for some  
$p_{i}\in S$, we need to check that the horizontal operator
\begin{equation}
   D^{H}_{\tE_{i}}: \ker D^{V}_{\tE_{i}}\to \ker D^{V}_{\tE_{i}}
\label{db.33}\end{equation} 
is invertible to insure that $D_{E}$ is Fredholm.  This will follow 
from the following lemma.
\end{proof}

\begin{lemma}
If $\alpha_{1}(p_{i})=0$, then the horizontal operator $D^{H}_{\tE_{i}}$ is
given by
\begin{equation}
    D^{H}_{\tE_{i}}= -\frac{1}{2}ic(du): \ker D^{V}_{\tE_{i}}\to \ker D^{V}_{\tE_{i}}.
\label{db.34}\end{equation}
In particular, it is invertible.
\label{db.35}\end{lemma}
\begin{proof}
Since $\tE$ is flat, the proof is essentially the same as in (\cite{Albin-Rochon2}, Proposition 3.1) with $\ell=0$.  Notice first
that the bundle on which $D_{E}$ acts is $(\underline{\bbC}^{k}\otimes {}^{\hc}\Lambda^{0,1}_{\tSigma}) \otimes \tE$.  Choose a spin structure on 
$\tSigma$ and let $S$ be the corresponding spinor bundle with respect to
the metric $g_{\Sigma}$.  Seen as a complex line bundle, $S$ is a 
square root of the canonical line bundle,
\begin{equation}
    S\otimes_{\bbC} S= K
\label{db.36}\end{equation}
Moreover, we have that
\begin{equation}
    (\underline{\bbC}\oplus {}^{\hc}\Lambda^{0,1}_{\tSigma}) \cong S\otimes_{\bbR}S^{*}
\label{db.37}\end{equation}
as real vector bundles.  Thus, the operator $D_{E}$ acts on
\[
     S\otimes_{\bbR}(S^{*}\otimes_{\bbC}\tE).
\]
As a bundle with connection, the bundle $S^{*}\otimes_{\bbC}\tE$ certainly
does not have a product structure near the boundary since it has 
non-zero curvature.  According to proposition 3.15 in \cite{Vaillant},
 the horizontal operator $D^{H}_{\tE_{i}}$ is given by
\begin{equation}
  (-iR) c\left(\frac{\pa}{\pa u}\right) 
\label{db.38}\end{equation} 
where $iR dg_{\Sigma}$ is the curvature of the complex vector bundle
$S^{*}\otimes \tE$, really the curvature of $S^{*}$ since $\tE$ is flat.
Since $S^{*}$ is a square root of $K^{-1}$, this means $R=\frac{1}{2}$ and 
the result follows.  
\end{proof}

For the operator $\db_{\cE}$, there is a similar discussion.  The restriction
$\tcE_{i}$ of the endomorphism bundle $\tcE$ to $\pa\tSigma_{p_{i}}$ can be described
as the quotient
\begin{equation}
   \bbZ\setminus \left(\bbR\times M_{k\times k}(\bbC)\right)
\label{db.39}\end{equation} 
where $\bbR\times M_{k\times k}(\bbC)$ is the total space of
the trivial bundle of $k\times k$ complex matrices over $\bbR$ with
$\bbZ$ action given by
\begin{equation}
   \begin{array}{lccc}
     m: & \bbR\times M_{k\times k}(\bbC) & \to & \bbR\times M_{k\times k}(\bbC)\\       & (u,A) & \mapsto & (u+m, \rho(S_{i})^{-m}A \rho(S_{i})^{m})
   \end{array}
\label{db.40}\end{equation}
for $m\in \bbZ$.  The identification 
\begin{equation}
 -c\left( \frac{d\rho_{\Sigma}}{\rho_{\Sigma}}\right): \left.
{}^{\hc}\Lambda^{0,1}_{\Sigma}\right|_{\pa \tSigma_{p_{i}}}\otimes \tcE_{i}\to
\tcE_{i}
\label{db.41}\end{equation} 
given by Clifford multiplication then allows one to write the vertical 
operator $D^{V}_{\tcE_{i}}$ as 
\begin{equation}
  D^{V}_{\tcE_{i}}=\left( \begin{array}{cc} 0 & \db^{V}_{\tcE_{i}}\\
                                \db^{V}_{\tcE_{i}} & 0 
                   \end{array}          \right)      
                                =\left(
  \begin{array}{cc}
    0 & i\frac{\pa}{\pa u} \\
   i\frac{\pa}{\pa u} & 0
   \end{array}
\right)\in \Psi^{1}(\pa\tSigma_{p_{i}}; \tcE_{i}\oplus \tcE_{i}).
\label{db.42}\end{equation}

\begin{lemma}
The spectrum of the vertical operator $\db^{V}_{\tcE_{i}}$ is given
by 
\[
   \lambda_{k}(j,l)= 2\pi( k+ \alpha_{j}(p_{i})-\alpha_{l}(p_{i}))
\quad \mbox{with multiplicity} \; k_{j}(p_{i})k_{l}(p_{i})
\]
for $k\in \bbZ$ and $j,l\in \{1,\cdots, r_{i}\}$.
\label{db.43}\end{lemma}
\begin{proof}
In terms of the decomposition \eqref{db.29}, we have the decomposition
\begin{equation}
M_{k\times k}(\bbC)= \bbC^{k}\otimes (\bbC^{k})^{*}= \bigoplus_{j,l=1}^{r_{i}}
   W_{ij}\otimes W_{il}^{*}
\label{db.44}\end{equation}
into the eigenspaces of the adjoint action of $\rho(S_{i})$ on
$M_{k\times k}(\bbC)$.  Here, the eigenspace $W_{ij}\otimes W^{*}_{il}$ has
corresponding eigenvalue $e^{2\pi i(\alpha_{j}-\alpha_{l})}$.  With respect to
these eigenspaces, the vertical operator decomposes as
\begin{equation}
  \db^{V}_{\cE_{i}}= \bigoplus_{j,l=1}^{r_{i}} \db^{V}_{\hom(\tE_{il};\tE_{ij})}
\label{db.45}\end{equation}  
with $\db^{V}_{\hom(\tE_{il};\tE_{ij})}$ acting on sections of 
\begin{equation}
\hom(\tE_{il},\tE_{ij}):= \bbZ\setminus (\bbR\times (W_{ij}\otimes W_{il}^{*})).
\label{db.45b}\end{equation}
If $f_{1},\ldots,f_{m}$ form a basis  of $W_{ij}\otimes W_{il}^{*}$,
then 
\begin{equation}
  e^{ijl}_{kp}:= e^{-2\pi i(k+\alpha_{j}-\alpha_{l})u}
f_{p}, \quad p\in \{1,\ldots,m\},\; k\in\bbZ
\label{db.46}\end{equation}
will be a $\db^{V}_{\hom(\tE_{il},\tE_{ij})}$-eigenbasis of $\Ld^{2}(\pa\tSigma_{p_{i}}, \hom(\tE_{il};\tE_{ij}))$ with $e^{ijl}_{kp}$ having eigenvalue
$\lambda_{k}(j,l)$, from which the result follows.
\end{proof}

\begin{proposition}
The operator $\db_{\cE}$ is Fredholm.
\label{db.47}\end{proposition}
\begin{proof}
As in lemma~\ref{db.35}, one computes that the horizontal operator of $D_{\cE}$ at 
the point $p_{i}$ is given by
\begin{equation}
 D^{H}_{\tcE_{i}}= -\frac{i}{2} c(du): \ker D_{\tcE_{i}}^{V}\to 
\ker D^{V}_{\tcE_{i}}.
\label{db.48}\end{equation}
In particular, it is clearly invertible and the result follows from 
proposition~\ref{db.22}.
\end{proof}

To end this section, let us compute the eta invariants of the self-adjoint 
operators $\db^{V}_{\tE_{ij}}$ and $\db^{V}_{\hom(\tE_{il},\tE_{ij})}$.

\begin{lemma}
The eta invariants of $\db^{V}_{\tE_{ij}}$ and $\db^{V}_{\hom(\tE_{il},\tE_{ij})}$ defined in \eqref{db.30} and 
\eqref{db.45}
are given by
\begin{equation*}
\begin{gathered}
    \eta(\db^{V}_{\tE_{ij}})=\left\{
   \begin{array}{ll} k_{j}^{i}(1- 2\alpha_{j}^{i}), & \alpha_{j}^{i}>0, \\
             0, & \alpha_{j}^{i}=0,
  \end{array} \right.
 \\ 
\eta(\db^{V}_{\hom(\tE_{il},\tE_{ij})})= \left\{ 
\begin{array}{ll}
k_{j}^{i}k_{l}^{i}\sign(\alpha_{j}^{i}-\alpha_{l}^{i})(1- 2|\alpha_{j}^{i}-\alpha_{l}^{i}|), & 
j\ne l, \\
 0, & j=l, 
\end{array}\right.
\end{gathered}
\end{equation*}
where $\alpha^{i}_{j}:= \alpha_{j}(p_{i})$ and $k^{i}_{j}= k_{j}(p_{i})$.
\label{db.49}\end{lemma}
\begin{proof}
The eta invariant of $\db^{V}_{\tE_{ij}}$ is the value at $s=0$ of the 
meromorphic extension of the eta functional
\begin{equation}
  \eta(\db^{V}_{E_{ij}},s):=  \sum_{\lambda \in \spec(\db^{V}_{E_{ij}})\setminus\{0\}} \frac{\lambda}{|\lambda|^{s+1}}, \quad \Re s>>1.
\label{db.50}\end{equation}
According to \eqref{db.31b} the spectrum of $\db^{V}_{E_{ij}}$ is symmetric
when $\alpha_{j}(p_{i})=0$, so in that case the eta invariant vanishes.
When $\alpha_{j}> 0$, we get instead for $\Re s \gg0$
\begin{equation}
\begin{aligned}
\eta(\db^{V}_{E_{ij}},s) &= \frac{k_{j}(p_{i})}{(2\pi)^{s}}
\sum_{k\in \bbZ}  \frac{k+ \alpha_{j}(p_{i})}{|k+ \alpha_{j}(p_{i})|^{s+1}}  \\
  &= \frac{k_{j}(p_{i})}{(2\pi)^{s}} \left(
      \sum_{k=0}^{\infty} \frac{1}{|k+\alpha_{j}^{i}|^{s}} - \sum_{k=1}^{\infty} 
       \frac{1}{|k-\alpha^{i}_{j}|^{s}} \right) \\
 &= \frac{k_{j}(p_{i})}{(2\pi)^{s}}( \zeta_{H}(s,\alpha_j(p_{i}))-
\zeta_{H}(s, 1-\alpha_{j}(p_{i})) )
\end{aligned}
\label{db.51}\end{equation}
where 
\begin{equation}
  \zeta_{H}(s,\beta)= \sum_{k=0}^{\infty} \frac{1}{|k+\beta|^{s}}, \quad
  \Re s > 1,
\label{db.52}\end{equation}
is the Hurwitz zeta function.  It admits an analytic continuation to 
$\bbC\setminus\{1\}$ and its value at $s=0$ is given by
\begin{equation}
  \zeta(0,\beta)= \frac{1}{2} -\beta, \quad \mbox{when} \; \beta>0.
\label{db.53}\end{equation}
Thus, this gives
\begin{equation}
\begin{aligned}
 \eta(\db^{V}_{E_{ij}})&= k_{j}(p_{i})( \zeta_{H}(0,\alpha_j(p_{i}))-
\zeta_{H}(0, 1-\alpha_{j}(p_{i})) )\\
&= k_{j}(p_{i})(1-2\alpha_{j}(p_{i})) 
\end{aligned}
\label{db.54}\end{equation}
as claimed.  For $\db^{V}_{\hom(E_{il},E_{ij})}$, the computation is 
very similar and relies on the knowledge of its spectrum 
described in \eqref{db.46}.  We leave the details to the reader.
\end{proof}

\section{Families of $\db$-operators over the moduli space}\label{ms.0}

For the moduli space $\cN$, a \textbf{universal parabolic stable vector bundle}
is a Hermitian vector bundle  $E\to \Sigma\times \cN$ such that 
for each $[\rho]\in \cN$, 
\begin{equation}
  \left. E\right|_{\Sigma\times\{[\rho]\}}\cong E_{\rho}
\label{if.1}\end{equation}
as Hermitian vector bundles.  From the point of view of representation
theory, the existence of a universal parabolic stable vector bundle is
equivalent to the existence of a smooth section
for the smooth principal $\PU(k)$-bundle
\begin{equation}
   \hom(\Gamma,\U(k))^{0}\to\hom(\Gamma,\U(k))^{0}/\U(k),
\label{if.2}\end{equation}  
where we recall $\hom(\Gamma,\U(k))^{0}$ is the space of irreducible admissible representations $\Gamma\to \U(k)$ for the prescribed weights and multiplicities with $\U(k)$ acting on this space by conjugation.  
According to (\cite{Boden-Yokogawa}, Proposition 3.2), a universal 
parabolic stable vector bundle exists for a generic weight system.  
From deformation theory, we know also that for $[\rho]\in \cN$, there 
exists a small neighborhood $\cU\subset \cN$ of $[\rho]$ such that
there exists a universal parabolic stable vector bundle on 
$\Sigma\times \cU$.  
From now on, we will assume either that the moduli space $\cN$ admits a universal
parabolic vector bundle $E\to \Sigma\times \cN$ or else, that we restrict $\cN$ to
an open set $\cU$ admitting a universal parabolic stable vector bundle.  Let 
\begin{equation}
  \sigma: \cN \to \hom( \Gamma, U(k))^{0}
\label{if.3}\end{equation}
be a smooth section and consider the induced universal parabolic vector bundle
$E\to \Sigma\times \cN$ such that 
\begin{equation}
 \left. E\right|_{\Sigma\times \{[\rho]\}}= E_{\sigma([\rho])}, \quad
[\rho]\in \cN.
\label{if.4}\end{equation}
On each fibre $\pi^{-1}([\rho])$ of the holomorphic fibration 
$\pi: \Sigma\times \cN\to \cN$, consider the operators
$\db_{E_{\sigma([\rho])}}$ and $\db_{\cE_{\sigma([\rho])}}$.  They
combine to give families of  $\db$ $\hc$-operators
\begin{equation}
\db_{E}\in \Psi^{1}_{\hc}(\Sigma\times \cN/\cN; E, {}^{\hc}
\Lambda^{0,1}_{\Sigma}\otimes E) ,\quad 
\db_{\cE}\in \Psi^{1}_{\hc}(\Sigma\times \cN/\cN; \cE, {}^{\hc}
\Lambda^{0,1}_{\Sigma}\otimes \cE)
\label{if.6}\end{equation} 
parametrized by $\cN$.  Since the fibration $\pi$ is trivial, it has 
a natural choice of connection, namely the trivial one.  Since $E$ is a Hermitian 
vector bundle, it has also a natural connection, namely its Chern connection.  
To describe it, we will first discuss the
equivalent of the Bers coordinates for the moduli space $\cN$ as introduced in \cite{TZ2}
(see also \cite{TZ89} for similar coordinates on the moduli space of stable vector bundles on
a compact Riemann surface).  

Recall first that the holomorphic tangent space $T_{[\rho]}\cN$ is naturally identified with the
space $\cH^{0,1}(\Sigma, \End(E_{\rho}))$ of square integrable harmonic $(0,1)$-forms on $\Sigma$
with value in $\End(E_{\rho})$.
From the point of view of index theory, this is intuitively clear.  An infinitesimal deformation
of the holomorphic structure of $E_{\rho}$ corresponds to an infinitesimal deformation of the 
operator $\db_{E_{\rho}}$, which amounts to adding an infinitesimal $(0,1)$-form 
$\nu\in \dot{\mathcal{C}}^{\infty}(\Sigma; \Lambda^{0,1}\Sigma\otimes \End(E_{\rho}))$,
\begin{equation}
   \db_{E_{\rho}} \to \db_{E_{\rho}}+ \nu,
\label{con.1}\end{equation}
where the symbol $\dot{\mathcal{C}}^{\infty}$ means we consider smooth sections with rapid 
decay as one approaches a puncture.
Notice however that if $\nu$ is in the image of $\db_{\cE_{\rho}}$, say $\nu= \db_{\cE_{\rho}} \mu$,
then $\nu$ is obtained from the infinitesimal reparametrization of $E$ given by
\begin{equation}
   \Id_{E_{\rho}}+ \mu: E_{\rho}\to E_{\rho}, \quad 
    \db_{E_{\rho}}+\nu= (\Id_{E_{\rho}}-\mu)\db_{E_{\rho}}(\Id_{E}+\mu)= \db_{E} +\db_{\cE_{\rho}}\mu.
\label{con.2}\end{equation}
In this case, the deformation $\nu$ leads to the same holomorphic structure up
to biholomorphism.  To get deformations leading to new holomorphic structures up to biholomorphism,
we need to mod out by the image of $\db_{\cE_{\rho}}$.  This means we can identify $T_{[\rho]}\cN$
with
\begin{equation}
 \coker \db_{\cE_{\rho}}\cong \cH^{0,1}(\Sigma;\End(E_{\rho})).
\label{con.3}\end{equation}   
The natural non-degenerate pairing
\begin{equation}
  \begin{array}{ccc}
  \cH^{0,1}(\Sigma;\End(E_{\rho}))\otimes \cH^{1,0}(\Sigma;\End(E_{\rho}))& \to & \bbC \\
     (\nu,\theta) & \mapsto & \int_{\Sigma} \tr( \nu\wedge \theta) 
     \end{array} 
\label{con.4}\end{equation}
allows one to identify $T^{*}_{[\rho]}\cN$ with the space $\cH^{1,0}(\Sigma;\End(E_{\rho}))$
of square integrable harmonic $(1,0)$-forms on $\Sigma$ with values in $\End(E_{\rho})$. 

Let $\rho: \Gamma\to U(k)$ be an admissible representation and suppose that 
$\rho= \sigma([\rho])$ where $\sigma$ is the smooth section \eqref{if.3} defining
the universal stable parabolic vector bundle $E\to \Sigma\times \cN$.  As pointed out
in \cite{TZ2}, for each $\nu\in \cH^{0,1}(\Sigma; \End(E_{\rho}))$ small enough, there exists
a unique map $f^{\nu}: \bbH\to \GL(k,\bbC)$ such that
\begin{enumerate}
\item[(i)] $\frac{\pa f^{\nu}}{\pa \overline{z}} (z)= f^{\nu}(z) \nu(z), \quad z\in \bbH$,
              where we also write $\nu$ for the lift of $\nu$ to the universal cover $\bbH$ of $\Sigma$;
\item[(ii)] $\rho^{\nu}(\gamma):= f^{\nu}(\gamma z) \rho(\gamma) f^{\nu}(z)^{-1}$ is independent
of $z\in \bbH$ and is an admissible irreducible unitary representation of $\Gamma$ with
  $\rho^{\nu}= \sigma([\rho^{\nu}])$.
\item[(iii)] $f^{\nu}$ is regular at the cusps, that is,
\[
        f^{\nu}(x_{i})= \lim_{z\to \infty} f^{\nu}(\sigma_{i}z) \in\bbC,\ i=1,\ldots,n.
\] 
\item[(iv)] $\det( f^{\nu}(z_{0}) )=1$ at $z_{0}=\sqrt{-1}\in \bbH$.   
\item[(v)] For $\epsilon\in [0,1)$, there are solutions $f^{\epsilon\nu}$ to (i),(ii),(iii) and
(iv) with $\nu$ replaced with $\epsilon \nu$ in such a way that $f^{0}=\Id$ is the identity
section and 
\[
         [0,1]\ni \epsilon\mapsto f^{\epsilon\nu}(z_{0})\in \GL(k,\bbC)
\] 
is a continuous map.
\end{enumerate}
For a fixed $\nu$, it not hard to see that there is at most one isomorphism class
$[\rho^{\nu}]\in \cN$ of irreducible admissible representations such that these properties
are satisfied by some $f^{\nu}$.  Since we require that $\rho^{\nu}=\sigma([\rho^{\nu}])$, one 
can also check that requirements (i)-(iii) determine the solution $f^{\nu}$ up to multiplication by an element
of $U(1)\subset U(k)$.  Requirement (iv) further forces
the solution to be unique up to multiplication by an element of $\bbZ_{k}\subset U(1)\subset U(k)$.  Since $\bbZ_{k}$ is a discrete subgroup of $U(k)$, the last requirement chooses canonically a unique solution among these $k$ possible solutions.

These solutions allow us to introduce complex coordinates near $[\rho]$ in $\cN$.  Precisely,
if $\nu_{1}, \ldots, \nu_{d}$ form a basis of $\cH^{0,1}(\Sigma,E_{\sigma([\rho])})$, then we get complex
coordinates
\begin{equation}
   \bbC^{d}\ni (\varepsilon_{1},\ldots,\varepsilon_{d})\mapsto [\rho^{\nu}]\in \cN,
\label{con.4b}\end{equation}
where $\nu= \varepsilon_{1}\nu_{1}+\cdots+ \varepsilon_{d}\nu_{d}$.  As mentioned
in \cite{TZ2}, the complex coordinates introduced in this way at two different points
$[\rho_{1}], [\rho_{2}]\in \cN$ transform holomorphically on overlaps, which induces on $\cN$
a complex structure.  

Now, for each $[\rho]\in \cN$, we can introduce these coordinates.  This allows us to
define a canonical connection on the universal stable parabolic vector bundle 
$E\to \Sigma\times \cN$.  On each fibre
\begin{equation}
  \left. E\right|_{\Sigma\times\{[\rho]\}}= E_{\sigma([\rho])},
\label{con.5}\end{equation}
the canonical connection restricts to the flat connection on $E_{\sigma([\rho])}$.  In the horizontal
direction at $(z,[\rho])\in \Sigma\times \cN$, we use the complex coordinates 
$\nu\in \cH^{0,1}(\Sigma,\End(E_{\rho}))$ introduced above and define
\begin{equation}
\begin{gathered}
  \nabla_{\nu}e (z,[\rho]):=  \left. \frac{\pa}{\pa \varepsilon}
       (f^{\varepsilon\nu})^{-1}e(z,[\rho^{\varepsilon\nu}])  \right|_{\varepsilon=0}, \\
       \nabla_{\overline{\nu}}e (z,[\rho]):=  \left. \frac{\pa}{\pa \overline{\varepsilon}}
       (f^{\varepsilon\nu})^{-1}e(z,[\rho^{\varepsilon\nu}])  \right|_{\varepsilon=0}.
\end{gathered}
\label{con.6}\end{equation}
These combine to give a canonical connection $\nabla^{E}$ for the universal stable
parabolic vector bundle $E$.  One can also check that together with the family $\db_{E}$, 
\eqref{con.6} induces a holomorphic structure on $E$.  This connection and holomorphic structure
also naturally induce a connection
$\nabla^{\cE}$ and a holomorphic structure on the endomorphism bundle $\cE=\End(E)$.

\begin{lemma}
The connections $\nabla^{E}$ and $\nabla^{\cE}$ are the respective Chern connections of the Hermitian bundles
$E\to \Sigma\times \cN$ and $\cE\to \Sigma\times \cN$.
\label{Chern.1}\end{lemma}

\begin{proof}
We will proceed as in the proof of lemma 1 in \cite{TZ89}.  Let 
$e_{1}(z,[\rho^{\nu}])$ and $e_{2}(z,[\rho^{\nu}])$ be local sections of $E$ near $(z_{0},[\rho])\in \Sigma\times\cN$ and 
set
\begin{equation}
\tilde{e}_{i}(z,[\rho^{\nu}]):= (f^{\nu})^{-1} e_{i}(z,[\rho^{\nu}]).
\label{Chern.2}\end{equation}
From 
\begin{equation}
\langle e_{1}(z,[\rho^{\nu}]), e_{2}(z,[\rho^{\nu}])\rangle_{E_{(z,[\rho^{\nu}])}}=
\langle \overline{f^{\nu}}^{\dag}f^{\nu}\tilde{e}_{1}(z,[\rho^{\nu}]), \tilde{e}_{2}(z,[\rho^{\nu}])\rangle_{E_{(z,[\rho])}},
\label{Chern.3}\end{equation}

we see that $\nabla^{E}$ will be the Chern connection of $E$ provided 
\begin{equation}
\Phi_{\nu}:= \frac{\pa }{\pa \varepsilon} \left. \overline{f^{\varepsilon\nu}}^{\dag} f^{\varepsilon \nu} \right|_{\varepsilon=0}=0
\label{Chern.4}\end{equation}
for each $[\rho]\in \cN$ and $\nu\in \cH^{0,1}(\Sigma; \End(E_{\sigma([\rho])})$.  From property (ii) of the definition of
$f^{\varepsilon \nu}$, we see that $\Phi_{\nu}\in \Omega^{0}(\Sigma,\End(E_{\sigma([\rho])}))$.
In fact, a simple computation shows that $\pa_{\overline{z}}\Phi_{\nu}=0$, so that 
$\Phi_{\nu}\in \ker \db_{\cE_{\sigma([\rho])}}$.  Since $E_{\sigma([\rho]}$ is stable, this
means that $\Phi_{\nu}$ is a multiple of the identity, and from the normalization condition
(iv) of the definition of $f^{\varepsilon \nu}$, we conclude that $\Phi_{\nu}=0$.

Similarly, the connection $\nabla^{\cE}$ will be the Chern connection of $\cE$ provided $\Ad \Phi_{\nu}=0$
for all $[\rho]\in \cN$ and $\nu\in \cH^{0,1}(\Sigma,\End(E_{\sigma([\rho])}))$, which follows immediately from
\eqref{Chern.4}.
\end{proof}

Since the fibration $\Sigma\times \cN\to \cN$ is trivial, we can choose to put on it the
trivial connection.  with this choice and the connection $\nabla^{E}$, we can then take the 
covariant derivatives for a family of operators $A\in \Psi^{*}(\Sigma\times \cN/\cN;E)$ with
respect to horizontal directions on $\cN$.  For the family of operators $\db_{E}$, one can
compute explicitly that
\begin{equation}
\begin{array}{ll} 
  \nabla_{\nu}\db_{E}= \nu, & \quad \nabla_{\overline{\nu}}\db_{E}=0, \\
  \nabla_{\nu}\db^{*}_{E}=0, & \nabla_{\overline{\nu}}\db_{E}^{*}= -*\nu * 
 \end{array}
\label{con.7}\end{equation} 
at $[\rho]\in \cN$, where $\nu\in \cH^{0,1}(\Sigma,\End(E_{\sigma([\rho])}))$.
Similarly, for the family of operators $\db_{\cE}$, one computes (\cf p.127 in \cite{TZ2} )
\begin{equation}
\begin{array}{ll} 
  \nabla_{\nu}\db_{\cE}= \ad\nu, & \quad \nabla_{\overline{\nu}}\db_{\cE}=0, \\
  \nabla_{\nu}\db^{*}_{\cE}=0, & \nabla_{\overline{\nu}}\db^{*}_{\cE}= -*\ad(*\nu) 
 \end{array}
\label{con.8}\end{equation} 
at $[\rho]\in \cN$, where $\nu\in \cH^{0,1}(\Sigma,\End(E_{\sigma([\rho])}))$.

The canonical connection $\nabla^{E}$ also naturally extends to define a connection on 
$\tE\to \tSigma\times \cN$.  When we consider $\db_{E}$ and $\db_{\cE}$ as families
of $\hc$-operators, this leads to the following useful fact.

\begin{lemma}
The vertical families $D^{V}_{\tE_{i}}$ and $D^{V}_{\tcE_{i}}$ are parallel with respect to the
induced connection,
\[
   [\nabla^{\tE_{i}\oplus \tE_{i}},D^{V}_{\tE_{i}}]=0, \quad  [\nabla^{\tcE_{i}\oplus \tcE_{i}},D^{V}_{\tcE_{i}}]=0.
\]
\label{con.9}\end{lemma}
\begin{proof}
It suffices to show that any form $\nu$ in $\cH^{0,1}(\Sigma,\End(E_{\rho}))$ is necessarily 
of rapid decay as one approaches a cusp (in the coordinates \eqref{db.6}), for then the
results easily follows from \eqref{con.7} and \eqref{con.8}. 

To see that the forms in $\cH^{0,1}(\Sigma,\End(E_{\rho}))$ have the claimed behavior,
notice first that if $\nu$ is in $\cH^{0,1}(\Sigma,\End(E_{\rho}))$, then
$*\nu$ is in $\cH^{1,0}(\Sigma,\End(E_{\rho}))$.  In the complex coordinate $\zeta_{i}=e^{2\pi i \sigma^{-1}_{i}z}$ with $\zeta_{i}(p_{i})=0$ near a puncture $p_{i}\in \cSigma$, the holomorphic form
$*\nu$ has the form
\begin{equation}
  *\nu(z) = \sum_{l,m=1}^{r_{i}}\sum_{k=0}^{\infty} a^{ilm}_{k}e^{2\pi i (\alpha_{l}(p_i)-\alpha_{m}(p_i))z}e^{2\pi i k z} dz
\label{con.10}\end{equation}
where $a^{ilm}_{k}\in W_{il}\otimes W^{*}_{im}$.  To insure that $*\nu$ is square integrable, we need that $a^{ilm}_{0}=0$ when $\alpha_{l}(p_{i})-\alpha_{m}(p_{i})\le 0$.  In particular, this implies that in the 
coordinates \eqref{db.6}, both $*\nu$ and $\nu=-*(*\nu)$  are of rapid decay
as one approaches a puncture.  
\end{proof}

\section{The index formul\ae}\label{if.0}
To compute the index formul\ae{}  of the families of operators $\db_E$ and $\db_{\cE}$ in \eqref{if.6},  we will use the results
of \cite{Albin-Rochon1}.  The main step consists in computing the 
eta forms of the vertical families of these operators.

For $i\in\{1,\ldots,n\}$ and $j\in\{1,\ldots,k_{i}\}$, let 
$E_{ij}\to \cN$ be the Hermitian vector bundle of rank $k_{j}(p_{i})$ 
with fibre at $[\rho]\in \cN$ given by the eigenspace corresponding to
the eigenvalue $e^{2\pi i\alpha_{j}(p_{i})}$ of 
$\sigma_{[\rho]}(S_{i})$.  For each $i$, consider the Hermitian vector bundle  obtained
by taking the direct sum
\begin{equation}
  E_{i}:= \bigoplus_{j=1}^{r_{i}} E_{ij}.
\label{if.5}\end{equation}
The bundle $E_{i}$ is clearly topologically trivial, but it is not necessarily trivial as a Hermitian vector bundle.
By lemma~\ref{con.9}, the Chern connection $\nabla^{E}$ induces a connection on each of the bundles $E_{ij}$.  This is
just the Chern connection $\nabla^{E_{ij}}$ of $E_{ij}$.   There is also an induced connection on $E_{i}$, namely its
Chern connection given by
\begin{equation}
      \nabla^{E_{i}} = \bigoplus_{j=1}^{r_{i}} \nabla^{E_{ij}}
\label{direct}\end{equation}      
\begin{lemma}
The (renormalized) eta form of the vertical family $\db_{E}^{V}$ is given by
\[
 \hat{\eta}(\db_{E}^{V})= \sum_{i=1}^{n} \sum_{\alpha_{j}(p_{i})>0}
    \lrpar{ \frac{1}{2}-\alpha_{j}(p_{i}) } \Ch(E_{ij}),
\]
where $\Ch(E_{ij}):= \Tr(\exp( \frac{i(\nabla^{E_{ij}})^{2}}{2\pi}))$ is a form representing the 
Chern character of $E_{ij}$.  
\label{if.9}\end{lemma}
\begin{proof}
Recall from \cite{BCO} (we will use the notation of 
\cite{Melrose-Piazza1}) that the eta form of the self-adjoint
family $\db_{E_{i}}^{V}$ is given by
\begin{equation}
  \eta(\db^{V}_{E_{i}})= \frac{1}{\sqrt{\pi}} \int_{0}^{\infty} 
 \STr_{\Cl(1)}\left(\frac{d\bbB_{t}}{dt} e^{-\bbB^{2}_{t}}    \right)dt
\label{if.10}\end{equation}
where $\bbB_{t}= t^{\frac{1}{2}}\bbB_{[0]}+ \bbB_{[1]}+ t^{-\frac{1}{2}}\bbB_{[2]}$ is the rescaled Bismut superconnection of the family 
$\db^{V}_{E_{i}}$.  It is given by 
\begin{equation}
\begin{gathered}
\bbB_{[0]}= \sigma \db^{V}_{E_{i}}, \\
\bbB_{[1]}= \sum_{\alpha} \varepsilon(f^{\alpha}) ( \nabla^{\tE_{i}\oplus \tE_{i}}+ \frac{1}{2} k^{\pa\tilde\pi}(f_{\alpha})), \\
\bbB_{[2]}= \sigma\frac{1}{4} \sum_{\alpha<\beta}
\varepsilon(f^{\alpha})\varepsilon(f^{\beta}) c(g_{i}) \Omega^{\pa\tilde{\pi}}(f_{\alpha},f_{\beta})(g_{i}). 
\end{gathered}
\label{if.11}\end{equation}
Here, $\{f_{\alpha}\}$ is a local orthonormal basis of the horizontal
tangent space, $\{g_{i}\}$ is a local orthonormal basis of the vertical
tangent space and $\sigma$ is the matrix 
\[
\sigma= \left(\begin{array}{cc} 0& 1 \\ 1 & 0 \end{array}\right).
\]
Since the fibration $\pa\tilde\pi:\pa\tSigma\times \cN\to \cN$ is trivial, 
the curvature $\Omega^{\pa\tilde\pi}$ and the 
second fundamental form $k^{\pa\tilde\pi}$ both vanish identically.
Thus, we have in fact
\begin{equation}
\bbB_{[0]}= \sigma \db_{\tE_{i}}^{V}, \quad \bbB_{[1]}=\sum_{\alpha} \varepsilon(f^{\alpha}) \nabla^{\tE_{i}\oplus \tE_{i}}, \quad \bbB_{[2]}=0. 
\label{if.12}\end{equation}

By lemma~\ref{con.9}, the family $\db^{V}_{E_{i}}$ is parallel with respect to the 
horizontal connection, $[\nabla^{\tE_{i}\oplus\tE_{i}}, \sigma\db^{V}_{E_{i}}]=0$,
which implies that
\[
  \bbB_{[0]}\bbB_{[1]}+ \bbB_{[1]}\bbB_{[0]}=0.
\]  
Consequently, $\bbB_{t}^{2}= t\bbB^{2}_{[0]}+ \bbB^{2}_{[1]}= t(\db^{V}_{E_{i}})^{2}+
\Omega^{\tE_{i}\oplus\tE_{i}}$.  The formula for the eta form therefore 
 simplifies to
\begin{equation}
\begin{aligned}
\eta(\db_{E_{i}}^{V}) &= \frac{1}{\sqrt{\pi}} \int_{0}^{\infty}
\STr_{\Cl(1)}\left(\frac{1}{2t^{\frac{1}{2}}} \sigma\db^{V}_{\tE_{i}}
 e^{-t(\db^{V}_{\tE_{i}})^{2}- \Omega^{\tE_{i}\oplus \tE_{i}}_{H}}   \right)dt \\
 & =\frac{1}{\sqrt{\pi}} \int_{0}^{\infty}
\Tr\left(\frac{1}{2t^{\frac{1}{2}}} \db^{V}_{\tE_{i}}
 e^{-t(\db^{V}_{\tE_{i}})^{2}- \Omega^{\tE_{i}}_{H}}   \right)dt.
\end{aligned}
\label{if.13}\end{equation}
where 
\[
  \Omega_{H}^{\tE_{i}}=\varepsilon(f^{\alpha})\varepsilon(f^{\beta})
\Omega^{\tE_{i}}(f_{\alpha},f_{\beta})
\]
is the horizontal contribution to the curvature of $\tE_{i}$.  It is such that
\[
  \Omega_{H}^{\tE_{i}}= \pa\tilde{\pi}^{*}\Omega^{E_{i}}
\]
where $\Omega^{E_{i}}$ is the curvature of the connection $\nabla^{E_{i}}$.
The fact that $\db^{V}_{\tE_{i}}$ is parallel with respect to the horizontal
connection also means that $\db^{V}_{\tE_{i}}$ commutes with 
$\Omega_{H}^{\tE_{i}}$.  In terms of the decomposition
\begin{equation}
 \db^{V}_{E_{i}}= \bigoplus_{j=1}^{r_{i}} \db^{V}_{E_{ij}},
\label{if.14}\end{equation}
this means that the eta form is given by
\begin{equation}
\eta(\db^{V}_{E})\sum_{j=1}^{r_{i}} \eta(\db^{V}_{E_{ij}})= 
\sum_{j=1}^{r_{i}} \frac{\eta(\db^{V}_{E_{ij}})_{[0]}}{k_{j}(p_{i})} 
\Tr\left(e^{-\Omega^{E_{ij}}}\right)
\label{if.15}\end{equation}
where $\Omega^{E_{ij}}= (\nabla^{E_{ij}})^{2}$.
The result then follows by using lemma~\ref{db.49} and summing over $i$ with the \textbf{renormalized} 
eta form given by (\cf \cite{BCO}, remark 4.101)
\begin{equation}
   \hat{\eta}(\db^{V}_{E}):=   \sum_{k} \frac{1}{(2\pi i)^{k}}  \eta(\db^{V}_{E})_{[2k]}.
\label{if.15b}\end{equation}
\end{proof}

\begin{theorem}
The Chern character of the families index of $\db_{E}$ is represented
in de Rham cohomology by the form
\begin{equation}
\pi_{*}(\Td(\Sigma)\Ch(E))- 
\sum_{i=1}^{n}\sum_{j=1}^{r_{i}}\left(\frac{1}{2}-\alpha_{j}(p_{i})\right)
\Ch(E_{ij})+ \sum_{\alpha_{1}(p_{i})=0} \Ch(E_{i1}).
\label{if.16}\end{equation}
\label{if.17}\end{theorem}

\begin{proof}
According to the index formula of (\cite{Albin-Rochon1}, theorem 4.5),   
the Chern character of the families index is represented by the form
\begin{equation}
\pi_{*}( \hat A(g_{\Sigma})\Ch'(E)) -\hat{\eta}(\db_{E}^{V})- \hat{\eta}(\db^{H}_{E}),
\label{if.18}\end{equation}
where $\Ch'(E)$ is the twisted Chern character of $E$ (see \cite{BGV} for a definition).
Since $T(\Sigma\times \cN)=T\Sigma\oplus T\cN$ is an orthogonal decomposition with respect
to the K\"ahler metric $g_{\Sigma}\oplus g_{\cN}$, there is a standard identification 
of forms 
\begin{equation}
\pi_{*}( \hat A(g_{\Sigma})\Ch'(E))= \pi_{*}(\Td(\Sigma) \Ch(E)).
\label{if.19}\end{equation}
Moreover, we computed the renormalized eta form of $\db^{V}_{\tE}$ in lemma~\ref{if.9}.  Thus,
it remains to compute the (renormalized) eta form of the horizontal family $\db^{H}_{\tE_{i}}$ 
for those $i$ such that $\alpha_{1}(p_{i})=0$.  From (\cite{Albin-Rochon1}, (4.12)),
we have,
\begin{equation}
 \hat{\eta}(\db^{H}_{E_{i}})= \frac{1}{2}\sign\lrpar{-\frac{1}{2}}\Ch(\ker \db^{V}_{E_{i}})=
 -\frac{1}{2} \Ch(\ker \db^{V}_{E_{i}}).
\label{if.20}\end{equation}
Since $\ker\db^{V}_{E_{i}}$ is canonically identified with  $E_{i1}$, this gives
\begin{equation}
 \hat{\eta}(\db^{H}_{E_{i}})=-\frac{1}{2} \Ch(E_{i1}).
\label{if.21}\end{equation}  
Combining this with the computation of the first two terms of \eqref{if.17}, the result follows.
\end{proof}

Similarly, we have an explicit formula for the Chern character of the index bundle of $\db_{\cE}$, starting with the following expression for its eta form.

\begin{lemma}
The renormalized eta form of $\db^{V}_{\tcE}$, $\hat{\eta}(\db^{V}_{\cE})$, is given by
\[
  \sum_{i=1}^{n}\sum_{j\ne l} \frac{\sign(\alpha_{j}(p_{i})-\alpha_{l}(p_{i}))(1- 2|\alpha_{j}(p_{i})-\alpha_{l}(p_{i})|)}{2}
   \Ch(E_{ij})\Ch(E_{il}^{*}). 
\]
\label{if.22}\end{lemma}
\begin{proof}
As in the proof of lemma~\ref{if.9}, the fibration $\pa \tilde\pi:\pa\tSigma\times \cN\to \cN$ 
is trivial and has curvature $\Omega^{\pa\tilde\pi}$ and second fundamental form $k^{\pa\tilde\pi}$ 
vanishing identically.  By lemma~\ref{con.9}, the vertical family is also parallel with respect to the connection
of $\pa\tilde\pi_{*}\tcE_{i}$.  Thus, using the decomposition
\begin{equation}
  \db_{\tcE_{i}}^{V}= \bigoplus _{j,l} \db^{V}_{\hom(\tE_{il},\tE_{jl})} 
\label{if.23}\end{equation}
and proceeding as in the proof of lemma~\ref{if.9}, we can write the eta form of the 
vertical family $\db^{V}_{\cE_{i}}$ at the $i$th boundary component as 
\begin{equation}
\begin{aligned}
\hat{\eta}(\db^{V}_{\cE_{i}})&= \hat{\eta}\left(\bigoplus_{j,l} \db^{V}_{\hom(\tE_{il},\tE_{ij})}\right), \\
 &=\sum_{j,l} \frac{\eta(\db_{\hom(E_{il},E_{ij})})_{[0]}}{k_{j}(p_{i})k_{l}(p_{i})} 
   \Ch( \hom(E_{il},E_{ij})).
\end{aligned}
\label{if.24}\end{equation} 
Since $\hom(E_{il},E_{ij})= E_{ij}\otimes E_{il}^{*}$ and $\eta(\db_{\hom(\tE_{il},\tE_{ij})})_{[0]}$
is $\frac{1}{2}$ the eta invariant of the family, we conclude from lemma~\ref{db.49} that 
\begin{equation}
\hat{\eta}(\db^{V}_{\tcE_{i}})= \sum_{j\ne l} \frac{\sign(\alpha_{j}(p_{i})-\alpha_{l}(p_{i}))(1- 2|\alpha_{j}(p_{i})-\alpha_{l}(p_{i})|)}{2}
   \Ch(E_{ij})\Ch(E_{il}^{*}).
\label{if.25}\end{equation}
Summing over $i$ to get the contributions from all the parabolic points, we get the result.      
\end{proof}

For a fixed fibre of the moduli space $\cN$, the identity section $\Id_{E_{\rho}}\in \End(E_{\rho})$ 
is obviously in the kernel of $\db_{\cE_{\rho}}$.  In fact, up to a constant, this is the
only element, since the eigenspaces of any other element of the kernel would decompose $E$ into parabolic subbundles, 
contradicting the stability of $E_{\rho}$.  
In particular, the dimension of the kernel does not jump as one moves on the moduli space and there is a well-defined kernel bundle
$\ker\db_{\cE}$ trivialized by the identity section $\Id_{\cE}\in \End(E)$.  Consequently,
there is also a cokernel bundle $\ker\db^{*}_{\cE}\to \cN$.  As above, the connection of $\pi_{*}\cE$ induces connections on $\ker \db_{\cE}$ and $\ker \db_{\cE}^*$.
With respect to the trivialization of the identity section $\Id_{\cE}$, this connection corresponds to the 
trivial connection.  Thus, in this context, the Chern character of the index bundle 
of the family $\db_{\cE}$ is represented by the form
\begin{equation}
  \Ch(\ker\db_{\cE}, \nabla^{\ker\db_{\cE}}) -
  \Ch(\ker\db^{*}_{\cE}, \nabla^{\ker\db^{*}_{\cE}})= 1 -
  \Ch(\ker\db^{*}_{\cE}, \nabla^{\ker\db^{*}_{\cE}})
\label{if.27}\end{equation}
\begin{theorem}
At the level of forms, the Chern character of the index bundle of $\db_{\cE}$ is
given by
\begin{equation*}
\begin{aligned}
1- \Ch(\ker\db^{*}_{\cE}, \nabla^{\ker\db^{*}_{\cE}})
&=  \pi_{*}(\Td(\Sigma) \Ch(\cE))  -\sum_{i=1}^{n}\sum_{j\ne l} \mu^{i}_{jl}
  \Ch(E_{ij})\Ch(E_{il}^{*}) \\
  &+ \frac{1}{2} \sum_{i=1}^{n} \sum_{l=1}^{r_{i}} \Ch(E_{il})\Ch(E_{il}^{*}) \\
  & -
  \left(\frac{1}{2\pi i}\right)^{\frac{N}{2}} d\int_{0}^{\infty} 
  \Str\left( \bbA^{t}_{D_{\cE}} e^{-(\bbA^{t}_{D_{\cE}})^{2}}\right)dt
  \end{aligned}
\end{equation*}
where $\bbA^{t}_{D_{\cE}}$ is the rescaled Bismut superconnection of 
$D_{\cE}= \sqrt{2}(\db_{\cE}+ \db_{\cE}^{*})$,
\[
 \mu^{i}_{jl}:= \sign(\alpha_{j}(p_{i})-\alpha_{l}(p_{i}))\left(\frac{1}{2}-|\alpha_{j}(p_{i})-\alpha_{l}(p_{i})|\right),
\]
and $N$ is the number operator for $\Lambda^{*}\cN$, that is, the action of $N$ on forms of degree $k$ on $\cN$ is multiplication by $k$.
\label{if.28}\end{theorem}
\begin{proof}
According to (\cite{Albin-Rochon1}, Theorem 4.5), we have the following formula,
\begin{multline}
1-\Ch(\ker\db_{\cE}^{*},\nabla^{\ker\db^{*}_{\cE}})= 
\pi_{*}( \hat A(\Sigma)\Ch'(\cE)) - \hat{\eta}(\db^{V}_{\cE}) -\hat{\eta}(\db^{H}_{\cE}) \\
 - \left(\frac{1}{2\pi i}\right)^{\frac{N}{2}} d\int_{0}^{\infty} 
  \Str\left( \bbA^{t}_{D_{\cE}} e^{-(\bbA^{t}_{D_{\cE}})^{2}}\right)dt.
\label{if.29}\end{multline}
As in theorem~\ref{if.17}, the first term on the right hand side can be rewritten 
in terms of the Todd form,
\[
\pi_{*}( \hat A(\Sigma)\Ch'(\cE))= \pi_{*}(\Td(\Sigma)\Ch(E)).
\]
The second term was computed in lemma~\ref{if.22}.  For the third term, we have according
to formula 4.12 in \cite{Albin-Rochon1},
\begin{equation}
\begin{array}{ll}
\hat{\eta}(\db^{H}_{\tcE_{i}})&=  \frac{1}{2} \sign\lrpar{-\frac{1}{2}} \Ch(\ker\db^{V}_{\tcE_{i}}) \\
                 &= -\frac{1}{2} \sum_{j=1}^{r_{i}} \Ch(E_{ij})\Ch(E^{*}_{ij}),  
\end{array}                 
\label{if.30}\end{equation}
since $\ker\db^{V}_{\tcE_{i}}$ is canonically identified with $\sum_{j=1}^{r_{i}} E_{ij}\otimes
E_{ij}^{*}$.  Summing over $i$, we get 
\begin{equation}
  \hat{\eta}(\db^{H}_{\cE})= -\frac{1}{2} \sum_{i=1}^{n}\sum_{j=1}^{r_{i}} \Ch(E_{ij})\Ch(E^{*}_{ij}).
\label{if.31}\end{equation}
\end{proof}

\section{The curvature of the determinant line bundle}\label{dlb.0}

In general, the geometry encoded in the Chern character of the index bundle is hard to unravel. The exception is the two-form part of the Chern character which is known to be the curvature of the determinant line bundle. This is true at the level of forms if these bundles are endowed with, respectively, the Bismut superconnection and the Quillen metric and connection (whose definition we now recall).

The determinant line bundle of a holomorphic family of Fredholm $\db$ operators $D_z$ 
is a line bundle over the parameter space, which at every point satisfies
\begin{equation}\label{DetBdle}
	(\Det D)_z 
	\cong (\Lambda^{\max} \ker D_{z})
	\otimes (\Lambda^{\max} \coker D_{z})^*.
\end{equation}
If the null spaces of $D_z$ fit together to form a vector bundle, then the right hand side of \eqref{DetBdle} serves as the definition of the determinant bundle. This is the case for instance for the family $D_{\cE}$.
If the the null spaces of $D_z$ do not form a bundle (e.g., when the dimension varies with $z$), but the spectrum of each $D_z$ is entirely made up of eigenvalues of finite multiplicity, then $\Det D$ can be constructed by a truncation procedure from   \cite{Quillen}, \cite{BF} (see also \cite{BGV}).
This is the case when the operators $D_z$ act on compact spaces; it is also the case for the family $D_E$ when all of the parabolic weights are non-zero by results of \cite{Vaillant}, as explained above.  

If the operators $D_z$ act on sections of a holomorphic bundle over a closed manifold, the line bundle $\Det D$ has a canonical choice of metric and connection defined using the zeta-regularized determinant of the family $D_z^*D_z$.
Recall that the zeta function of $D^*_zD_z$ is defined, for $\xi \gg 0$, by
\begin{equation*}
	\zeta_{D^*_zD_z} (\xi) = \frac1{\Gamma(\xi)} \int_0^{\infty} t^\xi \Tr(e^{-tD_z^*D_z} - \cP_{\ker D^*_zD_z} ) \frac{dt}t.
\end{equation*}
The short-time asymptotics of the heat kernel allow this function to be meromorphically continued to the whole complex plane.
The origin is a regular point of the extension and the derivative at the origin is used to define the determinant of $D_z^*D_z$ by
\begin{equation*}
	\log \det D_z^*D_z = - \zeta_{D^*_zD_z}'(0).
\end{equation*}
When $\ker D$ and $\coker D$ form actual bundles, the Quillen metric on the determinant line bundle is defined by starting with the $L^2$ metric induced from \eqref{DetBdle}, $\norm{\cdot}$, and then adjusting by the zeta-regularized determinant of $D_z^*D_z$,
\begin{equation}\label{QuillenNorm}
	\norm{\cdot}_Q = (\det D^*_zD_z)^{-1/2}\norm{\cdot}.
\end{equation}
When the nullspaces of $D_{z}$ do not form a bundle, the determinant line bundle does not have a well-defined induced
$L^{2}$-metric, but the Quillen metric does have a natural generalization which is globally well-defined. 
The Quillen metric is a Hermitian metric on the holomorphic line bundle $\Det D$ and hence has a unique compatible connection, the Chern connection. The curvature of this connection coincides with the two-form part of the Chern character of the Bismut superconnection on the index bundle of the family $D_z$, see for instance \cite[proposition~8.4]{Albin-Rochon2}.

If the prescribed weights are all non-zero, we have seen in proposition~\ref{db.32} that the operators
of the family $D_{E}$ all have discrete spectrum.  In fact, by the explicit construction of the heat
kernel given in \cite{Vaillant}, we also know that 
\begin{equation}
     e^{-t D^{2}_{_{E_{\rho}}}}  \in \rho^{\infty}_{\Sigma} 
     \Psi^{-\infty}_{b}(\tSigma; (\underline{\bbC}\oplus {}^{\hc}\Lambda^{0,1}_{\Sigma})\otimes \tE_{\rho})
\label{ds.1}\end{equation}
for $t>0$ and $[\rho]\in \cN$ ($\Psi_b$ denotes the space of $b$-pseudodifferential operators, see \cite{APSbook}).
In the family case, a similar statement holds for the heat kernel of the rescaled Bismut superconnection
$\bbA^{E}_{t}$.    In this non-compact context, the 
class of operators $\rho^{\infty}\Psi^{-\infty}_{b}(\tSigma)$ is really the analog of smoothing operators
on a compact manifold.  For instance, these operators are of trace class (while general operators in $\Psi^{-\infty}_b(\tSigma)$ are not).

The fact the spectrum of $D_{E_{\rho}}$ is discrete and its heat kernel satisfies
\eqref{ds.1}  indicates that  
the family $D_{E}$ spectrally behaves as a family of Dirac type operators on a compact manifold.  Because
of this, the standard definition of the Quillen metric and connection  and the computation of its curvature
for families of Dirac type operators acting on compact manifolds (as in \cite{Quillen}, \cite{BF} or \cite{BGV}) generalize almost immediately to the family $D_{E}$.  
The only
difference is that (see corollary~\ref{log.17} in Appendix~\ref{appendixA}) there are potentially extra powers of $\sqrt{t}$ involved in the asymptotic expansion of the trace of the heat kernel,
\begin{equation}
   \Tr(e^{-tD^{2}_{E_{\rho}}})= \frac{a_{-1}}{t} +
      \frac{a_{-\frac{1}{2}}}{\sqrt{t}} + a_{0} + \mathcal{O}(\sqrt{t}) \quad \mbox{as}\; t\to 0^{+}. 
\label{ds.2}\end{equation}
But the discussion in, for instance, \cite{BGV} works equally well with these extra asymptotic terms.  
This gives the following.

\begin{theorem}
When the weights of the parabolic structure are all non-zero, the curvature of the Quillen connection
$\nabla^{Q_{E}}$ 
associated to the determinant line of the family of operators $D_{E}$ is given by
\[
   \frac{i}{2\pi} (\nabla^{Q_{E}})^{2}= \pi_{*}(\Td(\Sigma)\Ch(E))_{[2]}-
        \sum_{i=1}^{n}\sum_{j=1}^{r_{i}}(\frac{1}{2}-\alpha_{j}(p_{i}))c_{1}(E_{ij})
\] 
\label{ds.3}\end{theorem}

One can also consider, instead of the family $\db_{E}$, the corresponding family $\db_{\overline{E}}$
on the compactified fibration $\cSigma\times\cN\to \cN$. In principle, this leads to a different determinant
line bundle since according to theorem~\ref{if.17}, $\db_{E_{\rho}}$ and $\db_{\overline{E}_{\rho}}$ have in general different indices.  
When the weights are non-zero and rational, Biswas and Raghavendra, in \cite{Biswas-Raghavendra},
computed the curvature of the determinant line bundle of the family $\db_{\overline{E}}$ defined
on the fibration $\cSigma\times\cN\to \cN$.  Their approach consisted in `unfolding' the parabolic
structure by lifting the family of operators $\db_{\overline{E}}$ to an appropriate cover $Y\to\cSigma$ 
ramified at the marked points $p_{1},\ldots,p_{n}$.    On this ramified cover $Y$, they could use
the idea of Quillen \cite{Quillen} to compute the curvature of the Quillen connection and relate
it to the natural symplectic form of the moduli space.

For the family of operators $\db_{\cE}$, there is also no problem defining the determinant bundle
since the kernel and cokernel of $\db_{\cE}$ form vector bundles and we can define $\Det \db_{\cE}$ directly by \eqref{DetBdle}. Moreover, since the bundle $\ker \db_{\cE}$ is trivial, we have
\begin{equation*}
	(\Det \db_{\cE})_z = (\Lambda^{\max} \coker (\db_{\cE})_z)^*.
\end{equation*}

The definition of the Quillen metric on $\Det \db_{\cE}$ is not as straightforward since their heat kernels need not be of trace class.
Indeed, recall that if $\cK_t(\zeta, \zeta')$ denotes the integral kernel of the heat kernel of $D_z^*D_z$, so that
\begin{equation*}
	e^{-tD_z^*D_z}f = \int \cK_t(\zeta, \zeta') f(\zeta') \; d\zeta',
\end{equation*}
then Lidskii's theorem says that the trace of $e^{-tD_z^*D_z}$ when it exists is given by
\begin{equation*}
	\Tr e^{-tD_z^*D_z} = \int \cK_t(\zeta, \zeta) \; d\zeta.
\end{equation*}
On a non-compact manifold, $\cK_t(\zeta, \zeta)$ will be a smooth function but need not be integrable.
This is the case for $\db_{\cE}^*\db_{\cE}$ for which we have a very precise description of the heat kernel from Vaillant's thesis \cite[\S 4]{Vaillant}.
Nevertheless, from this description we know that, if $x$ is a boundary defining function and $\eps>0$, then $ \int_{x\geq \eps} \cK_t(\zeta, \zeta)$ is finite and has an asymptotic expansion in $\eps$ so we can define
\begin{equation*}
	\RTr{ e^{-tD_z^*D_z} }
	= \sideset{^R}{}\int \cK_t(\zeta, \zeta) \; d\zeta 
	= \FP_{\eps=0} \int_{x\geq \eps} \cK_t(\zeta, \zeta) \; d\zeta.
\end{equation*}
This is a functional that coincides with the trace on operators of trace class, but that does not necessarily vanish on commutators.

The renormalized trace extends to $e^{-tD_z^*D_z} - \cP_{\ker D_z^*D_z}$ and so we define
\begin{equation*}
	{}^R\zeta_{\db_{\cE}^*\db_{\cE}} (\xi) 
	= \frac1{\Gamma(\xi)} \int_0^{\infty} t^\xi 
	\; {}^R\Tr(e^{-t\db_{\cE}^*\db_{\cE}} - \cP_{\ker \db_{\cE}^*\db_{\cE}} ) \frac{dt}t.
\end{equation*}
The description of the heat kernel in Vaillant's thesis also implies an expansion in $t$ as $t \to 0^+$ (see appendix~\ref{appendixA} below) which allows us to extend ${}^R\zeta_{\db_{\cE}^*\db_{\cE}}$ meromorphically to the whole complex plane and define
\begin{equation}
	\log \det \db_{\cE}^*\db_{\cE} = - {}^R\zeta_{D^*_zD_z}'(0)
\label{dlb.1a}\end{equation}
and then define $\norm{\cdot}_Q$ by \eqref{QuillenNorm}.

To define the Quillen connection of the determinant line bundle $\det \db_{\cE}$, we can take 
the Chern connection with respect to its Quillen metric.
However, to compute its curvature, it is better to use an alternative definition in terms of heat kernels.  
This require some preparation.  
On the moduli space $\cN$, we need to consider the 
Fr\'echet bundle $\pi_{*}\cE\to \cN$ whose fibre at $[\rho]\in \cN$ is given by
\begin{equation}
\pi_{*}\cE_{[\rho]}= \dot{\mathcal{C}}^{\infty}(\Sigma; \cE_{\sigma([\rho])}\otimes (
 \Lambda^{0,0}_{\Sigma}\oplus \Lambda^{0,1}_{\Sigma})\otimes |{}^{\hc}\Lambda_{\Sigma}|^{\frac{1}{2}})
\label{dlb.1}\end{equation}  
where $|{}^{\hc}\Lambda_{\Sigma}|^{\frac{1}{2}}$ is the half-density bundle on $\Sigma$.  Since the bundle $|{}^{\hc}\Lambda_{\Sigma}|$ is canonically 
trivialized by the volume form of the hyperbolic metric $g_{\Sigma}$, the families
of Dirac type operators 
\begin{equation}
D_{\cE}:= \sqrt{2}(\db_{\cE}+\db^{*}_{\cE}), \quad 
   D_{\cE}^{+}= \sqrt{2}\db_{\cE},\quad D_{\cE}^{-}= \sqrt{2}\db_{\cE}^{*}
\label{dlb.2}\end{equation}
naturally act on sections of $\pi_{*}\cE_{[\rho]}$.  The connection $\nabla^{\cE}$ described
in \eqref{con.6} then naturally induces a covariant derivative $\nabla^{\pi_{*}\cE}$ on
the Fr\'echet bundle $\pi_{*}\cE\to \cN$.  This allows one to define the rescaled superconnection
\begin{equation}
   \bbA^{s}_{\cE}:= s^{\frac{1}{2}}D_{\cE}+ \nabla^{\pi_{*}\cE}
\label{dlb.3}\end{equation}
and for $s\in\bbR^{+}$ the differential forms 
\begin{equation}
   \alpha^{\pm}(s):= {}^{R}\Tr_{\pi_{*}\cE^{\pm}}\left( \frac{\pa \bbA^{s}_{\cE}}{\pa s}
      e^{-(\bbA^{s}_{\cE})^{2}} \right).
\label{dlb.4}\end{equation}
As in equation (8.12) of \cite{Albin-Rochon2}, the integrals
\begin{equation}
\beta^{\pm}_{\cE}(z):= \int^{\infty}_{0} t^{z} \alpha^{\pm}_{\cE}(t)_{[1]}dt
\label{dlb.5}\end{equation} 
are holomorphic for $\Re z\gg 0$ and admit meromorphic extensions to the whole complex plane.
In particular, one can consider the $1$-forms 
\begin{equation}
  \beta^{\pm}_{\cE}:= \FP_{z=0}\frac{d}{dz} \left( \frac{1}{\Gamma(z)} \beta^{\pm}_{\cE}(z)  \right).
\label{dlb.6}\end{equation}
The orthogonal projections $P_{\pm}:\pi_{*}\cE^{\pm}\to \ker D^{\pm}_{\cE}$ induce connections
\begin{equation}
  \nabla^{\ker D^{\pm}_{\cE}}:= P_{\pm}\nabla^{\pi_{*}\cE^{\pm}}P_{\pm},
\label{dlb.7}\end{equation}
and so a connection $\nabla^{\det\db_{\cE}}$ on the determinant line bundle $\det\db_{\cE}$.  This
connection is holomorphic and is the Chern connection with respect to the $L^{2}$-metric
induced on $\det\db_{\cE}$.  The \textbf{Quillen connection} on $\det \db_{\cE}$ is defined
to be the connection  
\begin{equation}
   \nabla^{Q_{\cE}}:= \nabla^{\det \db_{\cE}}+ \beta^{+}_{\cE}.
\label{dlb.8}\end{equation}
\begin{proposition}
The Quillen connection is the Chern connection of $\det\db_{\cE}$ with respect to the Quillen
metric $\|\cdot\|_{Q_{\cE}}$.
\label{dlb.9}\end{proposition}
The proof of this proposition is essentially the same as the one of proposition 8.4 in
\cite{Albin-Rochon2}.  We will therefore not repeat it here, but simply point out that in
\cite{Albin-Rochon2}, lemma 8.1 was the key fact that allowed one to proceed as in the 
compact case (\cf \cite{BGV} and \cite{BF}).  In our context, the equivalent of this fact
is the following.
\begin{lemma}
The Schwartz kernel of $[\nabla^{\pi_{*}\cE},D^{\pm}_{\cE}]$ vanishes to all orders at the
front face.  In particular, for $P\in \Psi^{-\infty}((\Sigma\times \cN)/\cN; \cE\otimes (
 \Lambda^{0,0}_{\Sigma}\oplus \Lambda^{0,1}_{\Sigma})\otimes |\Lambda_{\Sigma}|^{\frac{1}{2}})$ a smooth family of smoothing operators on $\Sigma$ parametrized by $\cN$,
\[
   {}^{R}\STr\left( [[\nabla^{\pi_{*}\cE},D^{\pm}_{\cE}],P]   \right)=0.
\] 
\label{dlb.10}\end{lemma}
\begin{proof}
The fact that $[\nabla^{\pi_{*}\cE},D^{\pm}_{\cE}]$ vanishes to all orders at the front face
follows from \eqref{con.7} and the fact that a form $\nu$ in $\cH^{0,1}(\Sigma,\End(E_{\sigma([\rho])}))$ is necessarily of rapid decay as one approaches a
puncture (see the proof of lemma~\ref{con.9}).  Now, it is well-known that  
\[
   {}^{R}\STr\left( [[\nabla^{\pi_{*}\cE},D^{\pm}_{\cE}],P]\right) 
\] 
only depends linearly on the asymptotic expansion of the Schwartz kernels of $P$ and
 $[\nabla^{\pi_{*}\cE},D^{\pm}_{\cE}]$ at the corner of $\tSigma\times \tSigma$.  The 
 asymptotic expansion of $[\nabla^{\pi_{*}\cE},D^{\pm}_{\cE}]$ being trivial, the result
 follows. 
\end{proof}
With this lemma, the proof of proposition~\ref{dlb.9} is essentially the same as the one 
of proposition 8.1 in \cite{Albin-Rochon2}.  
We refer to \cite{Albin-Rochon2} for further details.  
We can now proceed as in the closed case (\cf \cite{BF} and \cite{BGV}, see
also \cite{Piazza} for the $b$-case) to compute the curvature of the Quillen connection.

\begin{theorem} 
The curvature of the Quillen connection on the determinant line bundle of $\db_{\cE}$
is given by
\begin{multline*}
  \frac{i}{2\pi} (\nabla^{Q_{\cE}})^{2}= 
   \pi_{*}(\Td(\Sigma)\Ch(\cE))_{[2]}  \\
   - \sum_{i=1}^{n}\sum_{j\ne l} \sign(\alpha_{j}(p_{i})-\alpha_{l}(p_{i}))(1-2|\alpha_{j}(p_{i})
   -\alpha_{l}(p_{i})|)k_{l}(p_{i})c_{1}(E_{ij}).
\end{multline*}
where $c_{1}(E_{ij}):= \frac{i}{2\pi}\Tr((\nabla^{E_{ij}})^{2})$ is the first Chern form
of $E_{ij}$ with respect to its naturally induced connection $\nabla^{E_{ij}}$.
\label{dlb.11}\end{theorem}

\begin{proof}
As in the case where the fibres are closed manifolds and as in theorem 8.5 of \cite{Albin-Rochon2}, the curvature of the Quillen connection is simply the two form part
of the Chern character of the index bundle given by theorem~\ref{if.28}, except for the exact term,
\[
-\left(\frac{1}{2\pi i}\right)^{\frac{N}{2}} d\int_{0}^{\infty} 
  \Str\left( \bbA^{t}_{D_{\cE}} e^{-(\bbA^{t}_{D_{\cE}})^{2}}\right)dt
\]
which does not contribute.  Putting all the terms involving $c_{1}(E_{ij})$ together for
$i\in \{1,\ldots,n\}$ and $j\in\{1,\ldots,r_{i}\}$, we get the desired result.
\end{proof}
Since $c_{1}(E_{ij})=c_{1}(\det E_{ij})$ is defined at the level of forms by using the Chern connection of $E_{ij}$ (\cf lemma 4
in \cite{TZ2}), 
our formula is the same as the one recently obtained by Takhtajan and Zograf in (\cite{TZ2}, theorem 2)\footnote{There is a typographical error in the statement of theorem 2 of \cite{TZ2}: the sum should be only for $l\ne m$.}.
In \cite{TZ2}, Takhtajan and Zograf also identified the term 
$\pi_{*}(\Ch(\cE))_{[2]}$ with the natural $(1,1)$-form of the moduli
space $\cN$,
\begin{equation}
   \pi_{*}(\Ch(\cE))_{[2]}= -\frac{1}{2\pi} \tilde{\Omega},
\label{dlb.15}\end{equation}
where 
\begin{equation}
\tilde{\Omega}\left( \frac{\pa}{\pa \varepsilon(\mu)}, \frac{\pa}{\pa \varepsilon(\nu)} \right)
= \frac{i}{2} \int_{\Sigma} \Tr\left(\ad \mu \wedge \ad *\nu \right).
\label{dlb.16}\end{equation}
This allowed them to use their formula to give a new way of computing the symplectic volume form of the moduli space $\cN$ in some special cases.

  The fact we get the same formula as in \cite{TZ2} is certainly expected at
the cohomological level,  but is not so trivial at the level of forms.  This is because {\em a priori}, a different definition
of the Quillen metric and Quillen connection is used in \cite{TZ2}.  Namely, the regularized determinant
\eqref{dlb.1a} is replaced by
\begin{equation}
  \det_{TZ}( \Delta_{E_{\rho}}):= \left.\frac{\pa}{\pa s}\right|_{s=1} Z(s,\Gamma;\Ad \rho)
\label{tz.1}\end{equation} 
where $Z(s,\Gamma;\Ad \rho)$ is the Selberg Zeta function associated to the operator
$\Delta_{E}$ (see \cite{TZ2} for more details and references).  

The fact we get the same formula as in \cite{TZ2} implies the following.

\begin{corollary}
Suppose that for a fixed weight system the moduli space $\cN$ is compact and admits a universal parabolic
stable vector bundle.  Then there is a constant $c_{\cN}>0$ such that 
\begin{equation}
       \det(D^{-}_{E}D^{+}_{E})= e^{-{}^{R}\zeta_{D^{-}_{E}D^{+}_{E}}'(0)}= c_{\cN} 
         \left.\frac{\pa}{\pa s}\right|_{s=1} Z(s,\Gamma;\Ad \rho).
\label{univ.1}\end{equation}
\label{tz2}\end{corollary} 
\begin{proof}
Denote by $\|\cdot\|_{TZ}$ the Quillen metric used in \cite{TZ2}.  Then there exists a smooth 
positive function $f: \cN\to \bbR$ such that 
\begin{equation}
     \|\cdot\|_{TZ}^{2}= f \|\cdot\|_{Q_{\cE}}^{2}.
\label{tz.3}\end{equation}
Since theorem~\ref{dlb.11} leads to the same formula as in theorem 2 of \cite{TZ2}, we have also that
\begin{equation}
    (\nabla^{Q_{\cE}})^{2}= (\nabla^{TZ})^{2}
\label{tz.4}\end{equation}
where $\nabla^{TZ}$ is the Chern connection associated to the Hermitian metric \newline $\|\cdot\|_{TZ}$.  
Now, recall that if $s:\cU\to \det(\db_{\cE})$ is a local holomorphic section of 
$\det(\db_{\cE})$, then the curvature of the Chern connections can be written as
\begin{equation}
  (\nabla^{TZ})^{2}= \db \partial \log\|s\|_{TZ}^{2}, \quad (\nabla^{Q_{\cE}})^{2}= \db \partial \log\|s\|_{Q_{\cE}}^{2}.
\label{tz.5}\end{equation}
Thus, combining \eqref{tz.3}, \eqref{tz.4} and \eqref{tz.5}, we get
\begin{equation}
    \db \partial \log f=0.
\label{tz.6}\end{equation}
Since we assume that $\cN$ is compact and since $\cN$ is connected (see for instance proposition 2.8 in
\cite{Nitsure}), we conclude by the maximum principle that $\log f$ and $f$ are
constant and the result follows with $c_{\cN}= f$.  

\end{proof}

\begin{remark}
As shown in \cite{Mehta-Seshadri}, see also \cite{Boden-Hu}, for a generic weight system, semi-stability
implies stability, which means in that case that the moduli space $\cN$ of stable
parabolic vector bundles is compact.  On the other hand, it is shown in \cite{Boden-Yokogawa} that for generic weight systems, the moduli space $\cN$ of stable parabolic vector bundles of vanishing parabolic degree admits a universal parabolic vector bundle.  Thus, corollary~\ref{tz2} can be reformulated as saying that 
\eqref{univ.1} holds for a generic weight system.   
\label{tz.7}\end{remark}

Of course, in the spirit of \cite{Sarnak} (see also \cite{Efrat1}, \cite{Efrat2}, \cite{Muller}, \cite{BJP}, \cite{Friedman} and \cite{Albin-Rochon2} for 
various generalizations in non-compact situations), Corollary~\ref{tz2} should be a direct consequence of a more general
result of the form
\begin{equation}
  \det(D^{-}_{E}D^{+}_{E} +s(1-s))= Z(s, \Gamma, \Ad \rho) G(s)
\label{detSel}\end{equation} 
for some universal meromorphic function $G(s)$ depending only on $g$, $n$ and the set of weights and
multiplicities. Presumably, the methods of \cite{Albin-Rochon2}, where similar regularized traces are used, could be 
adapted to this context to get a formula of the form \eqref{detSel}.   In particular,  one expects Corollary~\ref{tz2} to also hold when the moduli space is not compact.

\appendix
\section{Short time expansion of the trace of the heat kernel}\label{appendixA}

The presence of extra powers of $\sqrt{t}$  in the short-time expansion of the heat kernel in \eqref{ds.2} is easily explained using Vaillant's description of the heat kernel.
As it is no harder, and perhaps more interesting, we explain these asymptotics in the more general context of metrics with fibered hyperbolic cusps (or $\fD$ metrics) from \cite{Vaillant} and \cite{Albin-Rochon1}.
Recall that a Riemannian metric $g$ on the interior of a manifold with boundary $M$ is a (product-type\footnote{All of our considerations extend to the class of {\em exact} $\fD$ metrics.}) $\fD$ metric if:
\begin{quote}
1) The boundary is the total space of a fibration
\begin{equation*}
	Z - \pa M \xrightarrow{\phi} Y
\end{equation*}
2) There is a collar neighborhood of the boundary $\mathrm{Coll}(\pa M)$ such that for some choice of extension of $\phi$ to $\wt\phi:\mathrm{Coll} \to Y$ and a choice of connection for $\wt\phi$, $g$ is a submersion metric for $\wt\phi$ of the form
\begin{equation*}
	g\rest{\mathrm{Coll}(\pa M)} = \frac{dx^2}{x^2} + \phi^*g_Y + x^2 g_Z
\end{equation*}
where $x$ is a boundary defining function, $g_Y$ is a metric on $Y$, and $g_Z$ restricts to a metric on each fibre.
\end{quote}
We will denote the dimensions of $M$ and $Y$ by $n$ and $h$ respectively.

Recall that for a differential operator $D$, its heat kernel $e^{-tD}$ is the solution to the equation
\begin{equation*}
	\begin{cases}
	(\pa_t + D)e^{-tD} = 0 \\
	\displaystyle \lim_{t\to 0} e^{-tD} = \Id
	\end{cases}
\end{equation*}
The heat kernel acts by means of its Schwartz kernel, $\cK$, so that
\begin{equation*}
	e^{-tD}f(\zeta) = \int \cK(\zeta, \zeta', t) f(\zeta') \; d\zeta'
\end{equation*}
and, when $e^{-tD}$ is trace-class, Lidskii's theorem gives its trace as
\begin{equation*}
	\Tr( e^{-tD} ) = \int \cK\rest{\diag}.
\end{equation*}

Let $\eth$ be a Dirac-type operator associated to a $\fD$-metric and assume that the null space of its vertical family, $\eth^{V}:= \left. x\eth\right|_{\pa M}$, forms a bundle over $Y$ (we say that $\eth_{\fD}$ satisfies the constant rank assumption).  For such an operator, Vaillant found a very precise description of the Schwartz kernel of $e^{-t\eth^2}$ as a smooth function on the interior of a manifold with corners, $HM_{\fD}$ with asymptotic expansions at each of the boundary faces.
This construction is carried out in \cite[Chapter 4]{Vaillant} (see also \cite{Albin-Rochon1}).

For understanding the trace of the heat kernel, it is enough for us to recall what Vaillant's construction says about the restriction of the Schwartz kernel to the diagonal. The diagonal in $M^2$ pulled back to $M^2 \times \bbR^+$ can be identified with a submanifold of the interior of $HM_{\fD}$, whose closure we denote $\diag_H$. It is easy to describe $\diag_H$ directly without having to review the construction of $HM_{\fD}$. 
It is convenient to use $\sqrt t$ instead of $t$ since generally the heat kernel is smooth as a function of the former but not the latter -- for instance the Euclidean heat kernel is given by $(4\pi t)^{-n/2} \exp(-|x-y|^2/4t)$.
Thus we start with $M \times \bbR^+_s$, where $s = \sqrt t$, and introduce polar coordinates around the corner $\pa M \times \{0\}$.
We can do this geometrically by radially blowing-up $\pa M \times \{0\}$,
that is, we replace this submanifold with its inward-pointing unit normal bundle. The resulting manifold is denoted
\begin{equation*}
	\diag_H = [M \times \bbR^+_s; \pa M \times \{0\}]
\end{equation*}
and has three boundary faces: the boundary face $\bhs{\obf}$ coming from the lift of $\pa M \times \bbR^+$ to $\diag_{H}$, the temporal face $\bhs{\otf}$ coming from the lift of $M \times \{0\}$ to $\diag_{H}$, and the cusp face $\bhs{\ocf}$ coming from the blow-up of $\pa M \times \{0\}$.
\begin{figure}[htpb]
     \centering
  \includegraphics[height=1in]{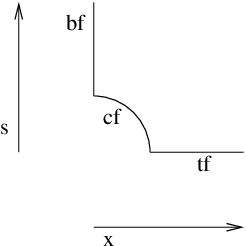}
   \caption{$\diag_H$}
    \label{DiagH}
\end{figure}
\phantom{We }
We will use $\rho_{\obf}$, $\rho_{\ocf}$ and $\rho_{\otf}$ to denote boundary defining functions for these faces.
From \cite[Lemma 5.26 (b)]{Vaillant} we know that the Schwartz kernel of $e^{-t\eth^2}$ satisfies
\begin{equation}\label{FromVail}
	\cK\rest{\diag_H} \in 
	\rho^{-1}_{\obf}\rho_{\ocf}^{-h}\rho_{\otf}^{-n+1} \CI( \diag_H, \Omega(\diag_H) ),
\end{equation}
where $\Omega(\diag_H)$ is the density bundle of $\diag_H$.  Thus, near $\bhs{\otf}\cap \bhs{\ocf}$, it is trivialized by the section $d\rho_{\otf}d\rho_{\ocf}$, while near $\bhs{\ocf}\cap\bhs{\obf}$, it is trivialized by the section $d\rho_{\ocf}d\rho_{\obf}$.
If the vertical family $\eth^{V}$ is invertible, we have in fact
\begin{equation}
   \cK\rest{\diag_H} \in
	\rho_{\obf}^{\infty}\rho_{\ocf}^{-h}\rho_{\otf}^{-n+1} \CI( \diag_H, \Omega(\diag_H) )
\label{log.1}\end{equation}
so that $e^{-t\eth^{2}}$ is trace class for positive time.
However, when the vertical family is not invertible, only \eqref{FromVail} holds and if $\cK_t$ is the restriction of $\cK$ to a fixed $t>0$, then $\cK_t\rest{\diag_H}$ is not integrable and hence $e^{-t\eth^2}$ is not of trace class.

One way to define the `trace' of the heat kernel in this case is to consider the function
\begin{equation*}
	z \mapsto 
	\Tr (x^z e^{-t\eth^2}) = \int_M x^z \cK_t \rest{\diag}.
\end{equation*}
For $\Re z$ large enough this is well-defined and for any $t>0$ it extends to be a meromorphic function of $z$ with at most simple poles.
We define the renormalized trace of the heat kernel to be the finite part of this function at $z=0$,
\begin{equation*}
	{}^R \Tr( e^{-t\eth^2} ) 
	= \sideset{^R}{_M}\int \cK_t \rest{\diag}
	= \FP_{z=0} \int_M x^z \cK_t \rest{\diag}.
\end{equation*}

As a function of $t$, $\cT = {}^R\Tr( e^{-t\eth^2} )$ inherits an asymptotic expansion as $t\to 0$ from the asymptotic expansions of $\cK$ at the boundary faces of $\diag_H$.
Each term in the expansion of $\cK$ at either $\bhs{\otf}$ or $\bhs{\ocf}$ gives rise to a term in the expansion of $\cT$ as $t \to 0$.
On the other hand, terms in the expansion of $\cK$ at the corners $\bhs{\otf} \cap \bhs{\ocf}$ and $\bhs{\obf}\cap\bhs{\ocf}$ potentially give rise to extra logarithmic terms in the expansion of $\cT$ as $t \to 0$.  More precisely, we have the
following. 
\begin{theorem}
Given  a Dirac-type operator $\eth$ associated to a $\fD$-metric with
vertical family $\eth^{V}$ satisfying the constant 
rank assumption, there exist constants $a_{k}, b_{k}$ for $k\in \bbN_{0}$ such that 
\begin{equation}\label{GralExpansion}
	{}^R\Tr( e^{-t\eth^2} ) \sim
	\lrspar{
	t^{-n/2} \sum_{k \geq 0} a_k t^{k/2}
	+ t^{-(h+1)/2} \sum_{k \geq 0} b_k t^{\frac{k}{2}} \log t } \; dt
\end{equation}
as $t\to 0$ .  
\label{log.2}\end{theorem}

\begin{proof}
Away from the corners, it is easy to see from
\eqref{FromVail} that the asymptotic expansions $\left.\cK\right|_{\diag_{H}}$ at the faces $\bhs{\ocf}$ and $\bhs{\otf}$ lead to asymptotic terms of the form
\[
     c_{k}t^{\frac{-n+k}{2}}, \quad c_{k}\in\bbC,\;, k\in\bbN_{0},
\]
in the small time asymptotic expansion of ${}^R\Tr( e^{-t\eth^2} )$. 
To understand how the logarithmic terms occur in the short time asymptotic, we 
need to study the contributions coming from the two corners $\bhs{\otf} \cap \bhs{\ocf}$ and $\bhs{\obf} \cap \bhs{\ocf}$.

If we were not renormalizing the integral, the expansion would follow directly from the push-forward theorem of \cite{Melrose1992}.
Furthermore, as pointed out in \cite[pg. 128]{Hassell-Mazzeo-Melrose}, if we were renormalizing using $\rho_{\obf}$ instead of $x$, the same theorem could be applied.
As it is, our situation is simple enough that we can proceed by direct computation (\cf example 3.2 in \cite{Grieser}).

By judicious choice of coordinate patches we can consider the two corners $\bhs{\otf} \cap \bhs{\ocf}$ and $\bhs{\obf} \cap \bhs{\ocf}$ separately.
First, near $\bhs{\otf}\cap \bhs{\ocf}$, we can choose the boundary defining functions to be given by 
\begin{equation*}
  \rho_{\ocf}=x, \quad \rho_{\otf}=\frac{\sqrt{t}}{x}.
\label{log.3}\end{equation*}
The reason this corner will contribute log-terms to the expansion is heuristically explained by looking at the level sets of $x$ near $\bhs{\otf} \cap \bhs{\ocf}$ in Figure \ref{DiagHLevel}, and is unrelated to renormalization.
\begin{figure}[htpb]
      \centering
      \includegraphics[height=1in]{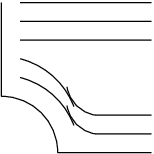}
      \caption{Lines we are integrating over}
      \label{DiagHLevel}
\end{figure}
To compute the contribution coming from the corner, we restrict our attention
to the small region $\cU_{\epsilon}$ defined by
\begin{equation*}
  \cU_{\epsilon}:= \{ p\in \diag_{H}\; | \; 0\le \rho_{\ocf}(p)\le \epsilon, 
  \; 0\le \rho_{\otf}(p)\le \epsilon \}.
\label{log.4}\end{equation*}
Thus, in the region $\cU_\epsilon$, we have $\frac{\sqrt{t}}{\epsilon}\le x\le \epsilon$.
Our choices of boundary defining functions give us a natural identification
\begin{equation*}
        \cU_{\epsilon}= \pa M\times [0,\epsilon]_{\rho_{\ocf}}\times [0,\epsilon]_{\rho_{\otf}}
\label{log.5}\end{equation*}
and a natural projection $\pi_{\pa M}:\cU_{\epsilon}\to \pa M$ onto the left factor.
From \eqref{FromVail}, we have that in the region $\cU_{\epsilon}$,
\begin{multline*}
\left.\cK\right|_{\cU_{\epsilon}}\in \rho^{-n+1}_{\otf}\rho_{\ocf}^{-h}\CI(\cU_{\epsilon};\pi_{\pa M}^{*}\Omega(\pa M))d\rho_{\ocf}d\rho_{\otf}\\
 = \rho^{-n}_{\otf}x^{-h-2}\CI(\cU_{\epsilon};\pi_{\pa M}^{*}\Omega(\pa M))dx dt.  
\label{log.6}\end{multline*}
We therefore have the following asymptotic expansion at the corner 
$\bhs{\otf}\cap\bhs{\ocf}$,
\begin{equation}
\begin{aligned}
\left.\cK\right|_{\cU_{\epsilon}}&\sim \sum_{k=-n}^{\infty}\sum_{\ell=-h-2}^{\infty}
    a_{k\ell}\rho^{k}_{\otf}x^{\ell} \; dx \; dt, \quad a_{k\ell}\in \CI(\pa M; \pi^{*}_{\pa M}\Omega(\pa M)), \\
&\sim  \sum_{k=-n}^{\infty}\sum_{\ell=-h-2}^{\infty}
    a_{k\ell}\sqrt{t}^{k}x^{\ell-k} \; dx \; dt.
\end{aligned}
\label{log.7}\end{equation}
If we write $\left. \cK\right|_{\cU_{\epsilon}}= a dxdt$ for some appropriate section $a$ of
$\pi_{\pa M}^{*}\Omega(\pa M)$, then
the integral of $\left. \cK\right|_{\cU_{\epsilon}}$ over the slice $\sqrt{t}=C$ can be written
\begin{equation*}
  \int_{\cU_{\epsilon}\cap \{t=C^{2}\}} \cK= 
    \left( \int_{\frac{C}{\epsilon}}^{\epsilon} \lrpar{ \int_{\pa M} a } dx\right) dt 
\label{log.8}\end{equation*} 
In particular, for the terms in the asymptotic expansion \eqref{log.7}, we get for
$\ell-k\ne -1$
\begin{multline*}
\left(\int_{\frac{C}{\epsilon}}^{\epsilon} \left(\int_{\pa M} a_{k\ell}\right)C^{k} x^{\ell-k}dx\right)dt=  \\
\left(\int_{\pa M}a_{k\ell}\right) \frac{1}{\ell-k+1}\left( C^{k}\epsilon^{\ell-k+1}- \frac{C^{\ell+1}}{\epsilon^{\ell-k+1}} \right) dt,
\label{log.9}\end{multline*}
while for $\ell-k=-1$, 
\begin{equation*}
  \left(\int_{\frac{C}{\epsilon}}^{\epsilon} \left(\int_{\pa M} a_{k\ell}\right)C^{k} x^{\ell-k}dx\right)dt=
  - \left(\int_{\pa M} a_{k\ell}\right) C^{k} \log\left(\frac{C}{\epsilon^{2}}\right).
\label{log.10}\end{equation*}
Thus, 
\begin{equation*}
    \int_{\cU_{\epsilon}\cap \{t=C^{2}\}} \left. \cK\right|_{\cU_{\epsilon}} 
	\sim C^{-n} \sum_{k=0}^{\infty} \alpha_{k} C^{k} 
	+ C^{-(h+1)} \sum_{k=0}^{\infty} \beta_{k} C^{k}\log C
\end{equation*}
when $C\searrow 0$, which gives an asymptotic expansion of the form given in \eqref{GralExpansion}.  

Next, near $\bhs{\obf}\cap \bhs{\ocf}$, we can choose the boundary defining functions to be given by 
\begin{equation*}
  \rho_{\ocf}=\sqrt t , \quad \rho_{\obf}=\frac{x}{\sqrt{t}}
\end{equation*}
Heuristically, from Figure \ref{DiagHLevel}, one would not expect the corner $\bhs{\obf} \cap \bhs{\ocf}$ to contribute log-terms to the expansion, however we shall see that the renormalization of the integrals causes these log-terms to appear.
We can consider the neighborhood 
\begin{equation*}
	\cV_{\epsilon} = \{ p\in \diag_{H} \; | \; 0\le \rho_{\ocf}(p),\rho_{\obf}(p) \le \epsilon\} 
\end{equation*}
of $\bhs{\obf}\cap\bhs{\ocf}$ in $\diag_{H}$.  Again, we 
have an identification $\cV_{\epsilon}=\pa M\times [0,1]_{\rho_{\ocf}}\times [0,1]_{\rho_{\obf}}$
and a natural projection $\pi_{\pa M}: \cV_{\epsilon} \to \pa M$ onto the left factor.  According to \eqref{FromVail}, we have
\begin{equation*}
	\left. x^{z}\cK\right|_{\cV_{\epsilon}} 
	\in \rho_{\obf}^{-1+z} \rho_{\ocf}^{-h-1+z} \CI(\cV_{\epsilon}; \pi_{\pa M}^{*}\Omega(\pa M))\; d\rho_{\obf}\; dt
\end{equation*}
with corresponding asymptotic behavior
\begin{equation*}
\begin{aligned}
	\left.x^{z}\cK\right|_{\cV_{\epsilon}} 
	&\sim 
	\sum_{k=-1}^{\infty}\sum_{\ell=-h-1}^{\infty}\tilde{a}_{k\ell}
	\rho^{k+z}_{\obf}\rho_{\ocf}^{\ell+z} \; d\rho_{\obf} \; dt, 
	\quad \wt a_{k\ell} \in \CI( \pa M; \Omega(\pa M) ), \\
	&\sim 
	\sum_{k=-1}^{\infty}\sum_{\ell=-h-1}^{\infty}\tilde{a}_{k\ell}
	x^{k+z} (\sqrt t)^{\ell-k-1} \; dx \; dt.
\end{aligned}
\end{equation*}
For each $k$ and $\ell$ and for $\Re z \gg |k|,|\ell|$, one computes that $\wt a_{k\ell}$ contributes to the asymptotic expansion of 
$\int_{\cV_{\epsilon} \cap \{t=C^{2}\} } \left.x^{z}\cK\right|_{\cV_{\epsilon}}$ via
\begin{equation*}
	\lrpar{ \int_0^{C\epsilon} \lrpar{ \int_{\pa M} \wt a_{k\ell} } C^{\ell-k-1} x^{k+z} \; dx } \; dt
	= \lrpar{ \int_{\pa M} \wt a_{k\ell} } \epsilon^{k+z+1} \frac{C^{\ell+z}}{k+z+1}  \; dt.
\end{equation*}
Finally taking the finite part at $z=0$, we get a contribution of
\begin{equation*}
	\begin{cases}
		\lrpar{ \int_{\pa M} \wt a_{k\ell} } \epsilon^{k+1} \frac{C^{\ell}}{k+1}  \; dt & \text{ if } k \neq -1 \\
		\lrpar{ \int_{\pa M} \wt a_{k\ell} } C^{\ell}\lrspar{\log \epsilon + \log C } \; dt 
			& \text{ if } k = -1
	\end{cases}
\end{equation*}
since $(C\epsilon)^{z}= e^{z\log(C\epsilon)}= 1+ z\log(C\epsilon) +\mathcal{O}(z^{2})$.
Thus we see that 
\begin{equation*}
	\int_{\cV_{\epsilon} \cap \{t=C^{2}\}} \left.x^{z}\cK\right|_{\cV_{\epsilon}} 
	\sim {C^{-(h+1)}} \left( \sum_{k=0}^{\infty} \tilde{\alpha}_{k} C^{k} + \sum_{k=0}^{\infty}\tilde{\beta}_k C^k \log C \right), 
\end{equation*}
which gives again an asymptotic expansion of the form \eqref{GralExpansion} and completes the proof.
\end{proof}

\begin{remark}
As on a closed manifold, one can show that the expansion at $\bhs{\otf}$ of $t^{n/2}$ times the heat kernel involves only powers of $t$ (instead of $\sqrt t$). 
\label{EvenHeat}
\end{remark}

\begin{remark}
As mentioned in the proof of the theorem, for integrable densities the pushforward theorem gives the form of the expansion \eqref{GralExpansion} with log-terms arising from the expansion at $\bhs{\otf} \cap \bhs{\ocf}$ but not from the corner $\bhs{\obf} \cap \bhs{\ocf}$.  As an example, for a  hyperbolic surface with cusps ($n=2$, $h=0$), we have an expansion
\begin{equation}
	{}^R\Tr( e^{-t\Delta_{\Sigma}} ) \sim
	\lrspar{
	t^{-1} \sum_{k \geq 0} a_k t^{k/2}
	+ t^{-1/2} \sum_{k \geq 0} b_k t^{\frac{k}{2}} \log t } \; dt
\label{log.18}\end{equation}
for the heat kernel of the Laplacian.  From remark~\ref{EvenHeat}, we know that the corner
$\bhs{\otf}\cap\bhs{\ocf}$ leads to no logarithmic term and an explicit computation at the corner 
$\bhs{\obf}\cap\bhs{\ocf}$ shows that $b_{0}\ne 0$ and is in fact the same as
the corresponding term in the short time expansion of the ``relative trace"
considered by M\"uller \cite[equation (2.3)]{Muller}.
\end{remark}

\begin{corollary}
For the operator $D_{E_{\rho}}$ considered in \eqref{ds.2}, we have the 
asymptotic expansion
\[
   \Tr(e^{-tD^{2}_{E_{\rho}}})\sim \frac{1}{t} \sum_{k=0}^{\infty} a_{k}t^{\frac{k}{2}}
\quad \mbox{as} \; t\searrow 0,
\]
for some constants $a_{k}$, $k\in \bbN_{0}$.
\label{log.17}\end{corollary}
\begin{proof}
From theorem~\ref{log.2}, we have an asymptotic expansion of the form \eqref{log.18}.
To see that the coefficient $b_{k}$ vanishes for all $k\in \bbN_{0}$, notice first that since the vertical operator of $D_{E_{\rho}}$ is invertible by assumption, we do not pick up any logarithmic terms from the corner $\bhs{\obf}\cap \bhs{\ocf}$ in light of \eqref{log.1}.  At
the other corner $\bhs{\otf}\cap \bhs{\ocf}$ the asymptotic expansion of 
$\left. \cK\right|_{\diag_{H}}$ is dictated by the standard local expansion
in the interior,
\begin{equation}
   \left. \cK\right|_{\diag_{H}}\sim \frac{1}{t} \sum_{k=0}^{\infty} c_{k} t^{k} \; dr \; d\theta dt
\label{log.19}\end{equation} 
in the coordinates of \eqref{db.6}.  Although a priori the coefficient
$c_{k}$ could depend on $r$ and $\theta$, it is in fact constant since
it is a universal expression in terms of the curvature of $E_{\rho}$, which is
zero, and the curvature of $g_{\Sigma}$, which is constant.  Since we would need a term of the form $t^{k}r^{-1}dr d\theta dt$ to pick up a logarithmic term,
we see that the asymptotic expansion of $\left.\cK\right|_{\diag_{H}}$ at 
the corner $\bhs{\otf}\cap \bhs{\ocf}$ leads to no logarithmic term.  Consequently, $b_{k}=0$ for all
$k\in \bbN_{0}$ and the result follows.
\end{proof}

\bibliography{ifpsb}
\bibliographystyle{amsplain}

\end{document}